\documentclass[leqno,11pt]{article}

\usepackage[T1]{fontenc}
\usepackage[english]{babel}
\usepackage{lmodern}

\usepackage[bookmarks=false]{hyperref}

\usepackage[intlimits]{amsmath}
\usepackage{amssymb, amsfonts}
\usepackage{amsthm}

\usepackage{mathbbol}
\usepackage{bbm}
\usepackage{mathrsfs}

\usepackage[all]{xy}

\usepackage{graphicx}
\usepackage{color}

\numberwithin{equation}{section}

\newtheorem{theoremcounter}{theoremcounter}[section]
\newtheorem{thmstarcounter}{thmstarcounter}

\newtheorem{corollary}[theoremcounter]{Corollary}
\newtheorem{lemma}[theoremcounter]{Lemma}

\newtheorem{proposition}[theoremcounter]{Proposition}
\newtheorem{theorem}[theoremcounter]{Theorem}
\newtheorem{thmstar}[thmstarcounter]{Theorem}

\theoremstyle{definition}

\newtheorem{definition}[theoremcounter]{Definition}

\newtheorem{remark}[theoremcounter]{Remark}

\newcommand{\cR}{\ensuremath{\mathcal{R}}}
\newcommand{\cS}{\ensuremath{\mathcal{S}}}

\newcommand{\cZ}{\ensuremath{\mathcal{Z}}}

\newcommand{\rG}{\ensuremath{\mathrm{G}}}

\newcommand{\rK}{\ensuremath{\mathrm{K}}}

\newcommand{\vphi}{\ensuremath{\varphi}}

\newcommand{\ol}{\overline}
\newcommand{\wt}{\widetilde}

\newcommand{\amid}{\ensuremath{\, | \,}}

\newcommand{\eqstop}{\ensuremath{\text{.}}}
\newcommand{\eqcomma}{\ensuremath{\text{,}}}

\newcommand{\NN}{\ensuremath{\mathbb{N}}}
\newcommand{\ZZ}{\ensuremath{\mathbb{Z}}}

\newcommand{\nisdiv}{\ensuremath{\mathop{\nmid}}}

\newcommand{\id}{\ensuremath{\mathrm{id}}}
\newcommand{\ra}{\ensuremath{\rightarrow}}

\newcommand{\hra}{\ensuremath{\hookrightarrow}}

\newcommand{\dom}{\ensuremath{\mathrm{dom}}}

\newcommand{\Sym}{\ensuremath{\mathrm{Sym}}}
\newcommand{\Aut}{\ensuremath{\mathrm{Aut}}}

\newcommand{\supp}{\ensuremath{\mathop{\mathrm{supp}}}}

\newcommand{\Ad}{\ensuremath{\mathop{\mathrm{Ad}}}}

\newcommand{\grpaction}[1]{\ensuremath{\stackrel{#1}{\curvearrowright}}}

\newcommand{\Stab}{\ensuremath{\mathrm{Stab}}}

\newcommand{\BS}{\ensuremath{\mathrm{BS}}}

\newcommand{\HNN}{\ensuremath{\mathrm{HNN}}}

\newcommand{\Prob}{\ensuremath{\mathrm{Prob}}}
\newcommand{\QN}{\ensuremath{\mathrm{QN}}}
\newcommand{\rng}{\ensuremath{\mathrm{rng}}}
\renewcommand{\dom}{\ensuremath{\mathrm{dom}}}

\renewcommand{\graph}{\ensuremath{\mathrm{graph}}}

\newcommand{\rp}{/ \hspace{-1mm} /}

\begin{document}
\begin{center}
  \textbf{\LARGE Baumslag-Solitar groups, relative profinite completions and measure equivalence rigidity}

\vspace{5mm}

  {\large by Cyril Houdayer\footnote{supported by ANR Grant NEUMANN} and Sven Raum\footnote{supported by KU Leuven BOF research grant OT/08/032}}
\end{center}

\begin{abstract}
\noindent
We introduce an algebraic invariant for aperiodic inclusions of probability measure preserving equivalence relations. We use this invariant to prove that every stable orbit equivalence between free pmp actions of direct products of non-amenable Baumslag-Solitar groups whose canonical subgroup acts aperiodically forces the number of factors of the products to be the same and the factors to be isomorphic after permutation. This generalises some of the results obtained by Kida in \cite{kida11-BS} and moreover provides new measure equivalence rigidity phenomena for Baumslag-Solitar groups. We also obtain a complete classification of direct products of relative profinite completions of Baumslag-Solitar groups, continuing recent work of Elder and Willis in \cite{elderwillis13}.
\end{abstract}

\section*{Introduction and statement of main results}
\label{sec:introduction}

The subject of {\em measured group theory} was introduced by Gromov in \cite{gromov93-asymptotic-invariants}.  It aims at studying discrete groups from a measure theoretic point of view. One of the fundamental problems in measured group theory is the classification of discrete groups up to {\em measure equivalence}. It is of particular interest to classify natural classes of discrete groups. Measured group theory has links to other active fields of mathematics such as von Neumann algebras \cite{popa07-cocycle-superrigidity, popa07-deformation-rigidity, popa08-spectral-gap}, geometric group theory \cite{shalom04, monodshalom06, kida06, kida09} and ergodic theory \cite{dye59, dye63, ornsteinweiss80,connesfeldmanweiss81, furman99, gaboriau00, gaboriau02}.  Surveys of current topics in measured group theory can be found in \cite{furman11,gaboriau11-measured-group-theory}.

The {\em Baumslag-Solitar groups} were introduced in  \cite{baumslagsolitar62} as the first examples of finitely presented non-Hopfian groups. Recall that for all $m, n \in \ZZ \setminus \{0\}$, the Baumslag-Solitar group $\BS(m, n)$ with parameters $(m, n)$ is defined by the presentation
\begin{equation*}
  \BS(m, n) = \langle a, t \amid t a^m t^{-1} = a^n\rangle
  \eqstop
\end{equation*}
The classification of the Baumslag-Solitar groups, up to isomorphism, was obtained in \cite{moldavanskii91}: $\BS(m, n) \cong \BS(p, q)$ if and only if there exists $\varepsilon \in \{-1, 1\}$ such that $\{m, n\} = \{\varepsilon p, \varepsilon q\}$. In this paper, we will always assume that $1 \leq |m| \leq n$. Note that for every $n \geq 1$, $\BS(\pm 1, n)$ is solvable, hence amenable and for every $2 \leq |m| \leq n$, $\BS(m, n)$ is non-amenable. The classification of the Baumslag-Solitar groups, up to quasi-isometry, was an important result in geometric group theory (see \cite{farbmosher98-BS-rigidity,whyte01-BS}). Recently, a partial classification of the von Neumann algebras of the Baumslag-Solitar groups was obtained in \cite{meesschaertvaes13}. 

The aim of the present paper is to obtain new rigidity results for all non-amenable Baumslag-Solitar groups and their direct products in the framework of measured group theory. Theorem~\ref{thm:main-result} provides  new stable orbit equivalence (SOE) rigidity results for actions of direct products of arbitrary non-amenable Baumslag-Solitar groups.

In \cite{kida11-BS}, Kida obtained several rigidity results for measure equivalence couplings of non-amenable Baumslag-Solitar groups. He showed in \cite[Theorem 1.4]{kida11-BS} that for a large class of non-amenable Baumslag-Solitar groups, any measurable coupling between such Baumslag-Solitar groups which is aperiodic on the natural cyclic subgroups, can only exist between Baumslag-Solitar groups which are isomorphic. Recall that a pmp action $\Gamma \grpaction{} (X, \mu)$ of a discrete countable group on a standard probability space is {\em aperiodic} if every finite index subgroup of $\Gamma$ acts ergodically on $(X, \mu)$. For instance, if the action $\Gamma \grpaction{} (X, \mu)$ is weakly mixing then it is aperiodic.

We will make use of the following notations: 
\begin{itemize}
\item For all $1 \leq i \leq k$, put $\BS(m_i,n_i) = \langle a_i, t_i \amid t_i a_i^{m_i} t_i^{-1} = a_i^{n_i} \rangle$.
\item For all $1 \leq j \leq l$, put $\BS(p_j,q_j) = \langle b_j, u_j \amid u_j b_j^{p_j} u_j^{-1} = b_j^{q_j} \rangle$.
\end{itemize}

\begin{thmstar}
  \label{thm:main-result}
 Let $k, l \in \NN \setminus \{0\}$. For every $i \in \{1, \dotsc, k\}$ and every $j \in \{1, \dotsc, l\}$, let $2 \leq |m_i| \leq n_i$ and $2 \leq |p_j| \leq q_j$.  Let
  \begin{gather*}
    \BS(m_1, n_1) \times \dotsm \times \BS(m_k, n_k) \grpaction{} (X, \mu) 
    \quad \text{and} \quad
    \BS(p_1, q_1) \times \cdots \times \BS(p_l, q_l) \grpaction{} (Y, \eta)
  \end{gather*}
  be stably orbit equivalent free ergodic probability measure preserving actions such that the actions $\langle a_1 \rangle \times \dotsm \times \langle a_k \rangle \curvearrowright (X, \mu)$ and $\langle b_1 \rangle \times \dotsm \times \langle b_l \rangle \curvearrowright (Y, \eta)$ are aperiodic.  
  
Then $k = l$ and there is a permutation $\sigma \in \Sym(k)$ such that  
for every $i \in \{1, \dotsc, k\}$, we have 
\begin{itemize}
\item $|m_i| = n_i = |p_{\sigma(i)}| = q_{\sigma(i)}$, if $|m_i| = n_i$ and
\item $m_i = p_{\sigma(i)}$ and $n_i = q_{\sigma(i)}$, if $|m_i| \neq n_i$.
\end{itemize} 
\end{thmstar}

When for every $i \in \{1, \dotsc, k\}$ and every $j \in \{1, \dotsc, l\}$, we have $2 \leq |m_i| < n_i$ and $2 \leq |p_j| < q_j$, the conclusion of Theorem  \ref{thm:main-result} is that $k = l$ and there is a permutation $\sigma \in \Sym(k)$ such that for every $i \in \{1, \dots, k\}$, $\BS(m_i, n_i) \cong \BS(p_{\sigma(i)}, q_{\sigma(i)})$. Note that in the case when $k = l = 1$, $2 \leq |m| < n$, $2 \leq |p| < q$, $m \nisdiv n$ and $p \nisdiv q$, Theorem \ref{thm:main-result} has already been proven in \cite[Theorem 1.4]{kida11-BS}. However, even in this particular case, our proof is new. As we will explain, the proof of Theorem  \ref{thm:main-result} relies on the classification of relative profinite completions of Baumslag-Solitar groups described in Theorem~\ref{thm:main-result-on-completions}.

Recall that since the natural subgroup $\langle a\rangle < \BS(m, n)$ is commensurated, one can associate to $\BS(m, n)$ a totally disconnected locally compact group denoted by $\rG(m, n)$ and called the {\em relative profinite completion} of $\BS(m, n)$ with respect to $\langle a\rangle$. The group $\rG(m, n)$ is obtained by taking the closure of the image of $\BS(m, n)$ in the Polish group $\Sym(\BS(m, n) / \langle a\rangle)$ when one views $\BS(m, n)$ acting by left multiplication on the set $\BS(m, n) / \langle a \rangle$.

The relative profinite completion of Baumslag-Solitar groups has recently been studied by Elder and Willis in \cite{elderwillis13}. Among other results, they showed that $\rG(m, n)$ `remembers' the rational number $|m|/n$ by computing the {\em scale function} of $\rG(m, n)$ \cite{willis94-structure}. We strengthen their results and moreover obtain a complete classification of direct products of relative profinite completions of arbitrary Baumslag-Solitar groups. We will use the following notation: for any topological groups $G$ and $H$, we write $G \cong H$ if there exists an isomorphism $\pi : G \to H$ such that $\pi$ and its inverse $\pi^{-1}$ are continuous.

\begin{thmstar}
  \label{thm:main-result-on-completions}
   Let $k, l \in \NN \setminus \{0\}$. For every $i \in \{1, \dotsc, k\}$ and every $j \in \{1, \dotsc, l\}$, let $2 \leq |m_i| \leq n_i$ and $2 \leq |p_j| \leq q_j$. Then
  \begin{equation*}
    \rG(m_1, n_1) \times \cdots \times \rG(m_k, n_k)
    \cong
    \rG(p_1, q_1) \times \cdots \times \rG(p_l, q_l)
  \end{equation*}
  if and only if $k = l$ and there is a permutation $\sigma \in \Sym(k)$ such that for every $i \in \{1, \dotsc, k\}$, we have $\rG(m_i, n_i) \cong \rG(p_{\sigma(i)}, q_{\sigma(i)})$. In this case, for every $i \in \{1, \dotsc, k\}$, we have 
\begin{itemize}
\item $|m_i| = n_i = |p_{\sigma(i)}| = q_{\sigma(i)}$, if $|m_i| = n_i$ and
\item $m_i = p_{\sigma(i)}$ and $n_i = q_{\sigma(i)}$, if $|m_i| \neq n_i$.
\end{itemize} 
\end{thmstar}

The proof of Theorem \ref{thm:main-result-on-completions} exploits the natural action of $\rG(m, n)$ on the Bass-Serre tree of $\BS(m, n)$ and the calculation of Willis's scale function \cite{willis94-structure} for relative profinite completions of Baumslag-Solitar groups obtained in \cite{elderwillis13}.  It is carried out in Section \ref{sec:products-of-BS}.

Let us briefly outline the proof of Theorem \ref{thm:main-result}. In Section~\ref{sec:invariant}, to any  aperiodic inclusion of ergodic pmp equivalence relations $\cS \subset \cR$ defined on a standard probability space $(X, \mu)$, we associate a non-Archimedian Polish group $H(\mathcal S \subset \mathcal R)$ which is an SOE invariant of the inclusion $\cS \subset \cR$. The group $H(\cS \subset \cR)$ is defined to be the essential image of an appropriate index cocycle for the inclusion $\cS \subset \cR$. When the inclusion $(\cS \subset \cR) = (\mathcal R(\Lambda \curvearrowright X) \subset \mathcal R(\Gamma \curvearrowright X))$ arises from a free ergodic pmp action $\Gamma \curvearrowright  (X, \mu)$ of a discrete countable group $\Gamma$ such that the subgroup $\Lambda < \Gamma$ is commensurated and the action $\Lambda \curvearrowright (X, \mu)$ is aperiodic, we show that the group $H(\mathcal R(\Lambda \curvearrowright X) \subset \mathcal R(\Gamma \curvearrowright X))$ is isomorphic to the relative profinite completion $\Gamma \rp \Lambda$. 

In Section \ref{sec:rigidity}, we show that any pmp equivalence relation $\mathcal R$ induced by a free ergodic pmp action $\BS(m_1, n_1) \times \dotsm \times \BS(m_k, n_k) \grpaction{} (X, \mu)$ of a direct product of non-amenable Baumslag-Solitar groups `remembers' the subequivalence relation $\mathcal S$ induced by the action $\langle a_1\rangle \times \cdots \times \langle a_k \rangle \curvearrowright (X, \mu)$. Therefore, combining the results from Sections \ref{sec:invariant} and \ref{sec:rigidity}, the proof of Theorem \ref{thm:main-result} reduces to the classification result for products of relative profinite completions of Baumslag-Solitar groups, namely Theorem \ref{thm:main-result-on-completions}.

\subsection*{Acknowledgments}

We are grateful to Mathieu Carette and Stefaan Vaes for their valuable comments.

\section{Preliminaries}
\label{sec:preliminaries}

\subsection{Totally disconnected locally compact groups}
\label{sec:locally-compact-groups}

Unless otherwise stated, all homomorphisms between topological groups that appear in this paper are assumed to be continuous. When dealing with isomorphisms between topological groups, we always assume that the isomorphism and its inverse are continuous.

A locally compact group $G$ is {\em totally disconnected} if the connected component $G^0$ of the neutral element equals $\{e\}$. By \cite{vandantzig36}, $G$ contains a compact open subgroup. Following \cite{willis94-structure}, the {\em scale function} $s_G : G \ra \NN \setminus \{0\}$ is defined by
\begin{equation*}
  s_G(g) = \min \{ [K : K \cap gKg^{-1}] \amid K \leq G \text{ compact open subgroup}\}
  \eqstop
\end{equation*}
By definition, the image of the scale function is an invariant of $G$. Moreover, if $\pi : G \to H$ is an isomorphism, then $s_G(g) = s_H(\pi(g))$ for every $g \in G$. 


\subsection{Baumslag-Solitar groups}
\label{sec:Baumslag-Solitar-groups}

If $m,n \in \ZZ \setminus \{0\}$, then $\langle a, t \amid t a^m t^{-1} = a^n \rangle$ is called the {\em Baumslag-Solitar group} with parameters $(m,n)$.  There are isomorphisms between Baumslag-Solitar groups with parameters $(m,n)$, $(n,m)$, $(-m,-n)$ and $(-n,-m)$.  So we will always assume that the parameters of a Baumslag-Solitar group satisfy $1 \leq |m| \leq n$.  Using this convention, the group $\BS(m,n)$ is amenable if and only if $m \in \{1, -1\}$.

A word in $w \in \BS(m,n)$ is called {\em freely reduced} if it is of the form $w = a^{i_1}t^{i_2}a^{i_3} \dotsm a^{i_{2k - 1}} t^{i_{2k}}a^{i_{2 k + 1}}$ with $k \geq 0$, $i_2, i_3, \dotsc, i_{2k - 1}, i_{2k} \in \ZZ \setminus \{0\}$ and $i_1, i_{2k + 1} \in \ZZ$.  A freely reduced word of the above form has no {\em pinches} if for every $l \in \{1, \dots, k - 1\}$, 
\begin{itemize}
\item $i_{2l} > 0$ and $i_{2(l+1)} < 0$ implies $m \nisdiv i_{2l + 1}$ and
\item $i_{2l} < 0$ and $i_{2(l+1)} > 0$ implies $n \nisdiv i_{2l + 1}$.
\end{itemize}
Britton's Lemma \cite{britton} says that no freely reduced word without pinches represents the neutral element of $\BS(m,n)$ unless $k = 0$ and $i_1 = 0$.

The Baumslag-Solitar group $\BS(m,n)$ is the HNN-extension $\HNN(\ZZ, m \ZZ, m \mapsto n)$ and hence acts on its {\em Bass-Serre tree} \cite{serre80}. Recall that the Bass-Serre tree $T$ of $\BS(m, n)$ is defined as follows: 
\begin{equation*}
  V(T) = \BS(m, n) / \langle a\rangle \ \text{ and } \ E(T)^+ = \BS(m, n) / \langle a^m\rangle
\end{equation*}
where $V(T)$ denotes the set of vertices of $T$ and $E(T)^+$ denotes the set of positive oriented edges of $T$. The {\em source} map $s : E(T)^+ \to V(T)$ and the {\em range} map $r : E(T)^+ \to V(T)$ are defined by
\begin{equation*}
  s(g \langle a^m\rangle) = g\langle a\rangle \ \text{ and } \ r(g \langle a^m\rangle) = g t^{-1} \langle a\rangle \ \text{ for all } \ g \in \BS(m, n)
  \eqstop
\end{equation*}
Then $\BS(m,n)$ naturally acts on $T$ by orientation-preserving simplicial automorphisms and without edge inversions. The action of $\BS(m,n)$ on $T$ is moreover both vertex and edge transitive.

\subsection{Relative profinite completions}
\label{sec:relative-profinite-completions}

Relative profinite completions were introduced in \cite{schlichting80}.  The present work makes essential use of this construction.  We recall the definition of relative profinite completions and some of its properties.

\begin{definition}
  \label{def:commensurable-subgroup}
  Let $\Lambda \leq \Gamma$ be an inclusion of discrete countable groups.  We say that $\Lambda$ is {\em commensurated} by $\Gamma$ if $g\Lambda g^{-1} \cap \Lambda$ has finite index in $\Lambda$ for every $g \in \Gamma$.
\end{definition}

Such a commensurated subgroup $\Lambda \leq \Gamma$ is also called a discrete {\em Hecke pair}. Let  $\Lambda \leq \Gamma$ and denote by $\Sym(\Gamma / \Lambda)$ the Polish group of all permutations of the countable set $\Gamma / \Lambda$. Let $\tau_{\Gamma, \Lambda}: \Gamma \ra \Sym(\Gamma / \Lambda)$ be the homomorphism induced by left multiplication.  If $\Lambda$ is commensurated by $\Gamma$, the action $\Lambda \curvearrowright \Gamma/\Lambda$ has finite orbits. Since moreover $g \Lambda \cap \Lambda = \emptyset$ whenever $g \in \Gamma \setminus \Lambda$, the subgroup $\ol{\tau_{\Gamma, \Lambda}(\Lambda)}$ is compact and open in $\ol{\tau_{\Gamma, \Lambda}(\Gamma)}$.
Hence $\overline{\tau_{\Gamma, \Lambda}(\Gamma)}$ is a totally disconnected locally compact group and the inclusion $\ol{\tau_{\Gamma, \Lambda}(\Lambda)} \leq \overline{\tau_{\Gamma, \Lambda}(\Gamma)}$ has countable index.

\begin{definition}
  \label{def:relative-profinite-completion}
  If $\Lambda \leq \Gamma$ is a commensurated subgroup, then $\ol{\tau_{\Gamma, \Lambda}(\Gamma)} = \Gamma \rp \Lambda$ is the {\em relative profinite completion} of $\Gamma$ with respect to $\Lambda$.
\end{definition}

Note that the relative profinite completion is a generalisation of a quotient of groups.  Indeed we have $\ker(\tau_{\Gamma, \Lambda}) \leq \Lambda$ and  $\Gamma / \ker(\tau_{\Gamma, \Lambda})$ injects into $\Gamma \rp \Lambda$. Note that $\Gamma \rp \Lambda$ is not strictly speaking a completion of $\Gamma$ unless $\ker(\tau_{\Gamma, \Lambda})$ is the trivial group.

We need the following two propositions about compatibility of relative profinite completions with respect to products and HNN-extensions.
\begin{proposition}
  \label{prop:relative-profinite-completions-of-products}
 Let $n \in \NN \setminus \{0\}$. For every $i \in \{1, \dotsc, n\}$, let $\Lambda_i \leq \Gamma_i$ be commensurated subgroups.  Then $\prod_i \Lambda_i \leq \prod_i \Gamma_i$ is commensurated and $(\prod_i \Gamma_i) \rp (\prod_i \Lambda_i) \cong \prod_i (\Gamma_i \rp \Lambda_i)$.
\end{proposition}
\begin{proof}
  The canonical bijection $(\prod_i \Gamma_i)/ (\prod_i \Lambda_i) \cong \prod_i (\Gamma_i / \Lambda_i)$ intertwines $(\prod_i \Gamma_i) \rp (\prod_i \Lambda_i)$ and $\prod_i (\Gamma_i \rp \Lambda_i)$.
\end{proof}

If $m, n \in \ZZ \setminus \{0\}$, then the natural subgroup $\langle a \rangle \leq \BS(m,n)$ is commensurated. We denote by $\rG(m,n)$ the relative profinite completion of $\BS(m,n)$ with respect to $\langle a \rangle$ and by $\rK(m,n)$ its natural compact open subgroup $\ol{\langle a \rangle}$. Denote by $T$ the Bass-Serre tree of $\BS(m, n)$. Observe that since the action of $\BS(m, n)$ on $\BS(m, n)/\langle a\rangle$ preserves adjacency, we have $\tau_{\BS(m, n), \langle a\rangle}(\BS(m, n)) < \Aut(T)$ and $\rG(m, n)$ coincides with the closure of $\tau_{\BS(m, n), \langle a\rangle}(\BS(m, n))$ in $\Aut(T)$. From now on, we will always regard $\rG(m, n)$ as a closed subgroup of $\Aut(T)$.

If $|m| = n$, then $\langle a^m\rangle$ is a normal subgroup in $\BS(m, n)$ and it has index equal to $m$ in $\langle a\rangle$. In that case, $\rG(m, n)$ is a discrete subgroup of $\Aut(T)$ and we have $\rG(m, n) \cong \ZZ/m\ZZ \ast \ZZ$ and $\rK(m, n) \cong \ZZ/m\ZZ$.

If $|m| \neq n$, then $\ker(\tau_{\BS(m, n), \langle a\rangle})$ is trivial. In that case, we regard $\BS(m, n)$ as a dense subgroup of $\rG(m,n)$. Moreover, $\rG(m, n)$ is not discrete.

We make use of the following proposition identifying $\rG(m,n)$ as an HNN-extension.  The HNN-extension of a locally compact group $L$ with open subgroups $H, K$ and an isomorphism $\vphi:H \ra K$ is by definition the abstract HNN-extension $\HNN(L, H, \vphi)$ equipped with the unique group topology defined by declaring the inclusion $L \hra \HNN(L, H, \vphi)$ an open homeomorphism onto its image.  It is easy to check that topological HNN-extensions in this sense satisfy the usual universal property of HNN-extensions with respect to continuous homomorphisms.

Write $m \rK(m, n) = \overline{\langle a^m \rangle}$ and $n \rK(m, n) = \overline{\langle a^n\rangle}$.

\begin{proposition}
  \label{prop:relative-profinite-completion-of-HNN}
 Let $1 \leq  |m| \leq n$. Then the inclusion $\rK(m, n) \leq \rG(m,n)$ is isomorphic to the inclusion $\rK(m, n) \leq \HNN(\rK(m,n), m \rK(m,n), m \mapsto n)$. 
\end{proposition}
\begin{proof}
 The result is clear when $|m| = n$ since in that case, we have $\rK(m, n) = \ZZ/m\ZZ$ and $\rG(m, n) = \ZZ/m\ZZ \ast \ZZ$.
 
 We assume that $|m| \neq n$. Let $T$ be the Bass-Serre tree of $\BS(m,n)$. Then $\rK(m,n) = \Stab_{\rG(m,n)}(\langle a \rangle)$ and $\rK(m,n) \cap t^{-1} \rK(m,n) t = \Stab_{\rG(m, n)}(\langle a^m\rangle)$. Observe that the action of $\rG(m, n)$ on $T$ is both vertex and edge transitive and $\rG(m,n) = \langle \rK(m,n) , t \rangle$. Then Bass-Serre theory implies that $\rG(m,n) \cong \HNN(\rK(m,n), \rK(m,n) \cap t^{-1} \rK(m,n) t, \Ad t)$ as abstract groups.  In order to see that this isomorphism and its inverse are continuous, it suffices to note that it is the identity on the open subgroup $\rK(m,n)$.

We have $m \rK(m,n) \subset \rK(m,n) \cap t^{-1} \rK(m,n) t$, since $\langle a^m \rangle$ is dense in $m\rK(m,n)$.  On the other hand, $\rK(m,n)$ acts transitively on $\{ t^{-1} \langle a \rangle, \dotsc, a^{m-1} t^{-1} \langle a \rangle \}$ and $\rK(m,n) \cap t^{-1} \rK(m,n) t = \Stab_{\rK(m, n)}(t^{-1} \langle a \rangle)$.  So $[\rK(m,n) : \rK(m,n) \cap t^{-1} \rK(m,n) t] = |m|$. Since moreover $[\rK(m, n) : m\rK(m, n)] \leq |m|$, we get that $\rK(m,n) \cap t^{-1} \rK(m,n) t = m \rK(m,n)$.  Similarly, one finds that $\rK(m,n) \cap t \rK(m,n) t^{-1} = n \rK(m,n)$.  Since $(\Ad t)|_{\langle a^m \rangle}$ equals the map $m \mapsto n$, continuity shows that $\rG(m,n) \cong \HNN(\rK(m,n), m \rK(m,n), m \mapsto n)$.
\end{proof}

Note that $\rK(m, n)$ is a profinite completion of the integer ring $\ZZ$ whose description can be found in \cite[Proposition 8.1]{elderwillis13}.

\subsection{Probability measure preserving equivalence relations}
\label{sec:pmp-equivalence-relations}

We recall some notions regarding probability measure preserving (pmp) equivalence relations. Unless otherwise stated, all sets and maps that appear in this paper are assumed to be Borel, and relations among Borel sets and maps are understood to hold, up to sets of measure zero.

Let $(X, \mu)$ be a standard probability space. A {\em probability measure preserving countable Borel equivalence relation} $\cR$ defined on the space $(X, \mu)$ is an equivalence relation such that the following properties hold:
\begin{itemize}
\item $\cR \subset X \times X$ is a Borel subset.
\item For $\mu$-almost every $x \in X$, the $\cR$-equivalence class $[x]_\cR$ is countable.
\item For every $\varphi \in [\cR]$, we have $\varphi_\ast \mu = \mu$.
\end{itemize}
In this paper, we simply say that $\mathcal R$ is a pmp equivalence relation. We denote by $[\cR]$ the {\em full group} of all Borel automorphisms $\varphi : X \to X$ such that $\graph(\varphi) \subset \mathcal R$. We also denote by $[[\cR]]$ the {\em full pseudogroup} of all partial Borel isomorphisms $\varphi : U \to V$ such that $\graph(\varphi) \subset \mathcal R$. We say that $\mathcal R$ is {\em ergodic} if every $\mathcal R$-invariant Borel subset $U \subset X$ has measure $0$ or $1$.

Following \cite{feldmanmoore77}, we define the $\sigma$-finite Borel measure $\nu$ on $\mathcal R$ in the following way. For every Borel subset $\mathcal V \subset \cR$, put
\begin{equation*}
  \nu(\mathcal V) = \int_X |\{y \in X \amid (x, y) \in \mathcal V\}| \, {\rm d}\mu(x)
  \eqstop
\end{equation*}
Since $\cR$ is assumed to be pmp, we also have $\nu(\mathcal V) = \int_X |\{x \in X \amid (x, y) \in \mathcal V\}| \, {\rm d}\mu(y)$. Then $(\mathcal R, \nu)$ is a standard measure space.

Likewise, define the Borel subset $\cR^{(2)} = \{(x, y, z) \in X \times X \times X \amid (x, y) \ \text{and} \ (y, z) \in \mathcal R\}$. We can then define the $\sigma$-finite Borel measure $\nu^{(2)}$ on $\cR^{(2)}$ in a similar way. For every Borel subset $\mathcal W \subset \cR^{(2)}$, put
\begin{equation*}
  \nu^{(2)}(\mathcal W) = \int_X |\{(y, z) \in \mathcal R \amid (x, y, z) \in \mathcal W\}| \, {\rm d}\mu(x)
  \eqstop
\end{equation*}
Since $\cR$ is assumed to be pmp, we also have
\begin{equation*}
  \nu^{(2)}(\mathcal W) = \int_X |\{(x, z) \in \mathcal R \amid (x, y, z) \in \mathcal W\}| \, {\rm d}\mu(y)  = \int_X |\{(x, y) \in \mathcal R \amid (x, y, z) \in \mathcal W\}| \, {\rm d}\mu(z)
  \eqstop
\end{equation*}
Then $(\mathcal R^{(2)}, \nu^{(2)})$ is a standard measure space.

For every non-negligible Borel subset $U \subset X$, put $\mu_{U} = \frac{1}{\mu(U)} \mu|_{U}$ and observe that $(U, \mu_{U})$ is a standard probability space. Define $\mathcal R |_{U} = \mathcal R \cap (U \times U)$. Then $\mathcal R |_{U}$ is a pmp equivalence relation defined on $(U, \mu_{U})$. Let now $\mathcal S$ be another pmp equivalence relation defined on a standard probability space $(Y, \eta)$.
We say that $\mathcal R$ and $\mathcal S$ are 
\begin{itemize}
\item {\em orbit equivalent} if there exists a pmp Borel isomorphism $\varphi : (X, \mu) \to (Y, \eta)$ such that ${(\varphi \times \varphi)(\cR)} = \cS$, and
\item {\em stably orbit equivalent} if there exist non-negligible Borel subsets $U \subset X$, $V \subset Y$ and a pmp Borel isomorphism $\varphi : (U, \mu_U) \to (V, \eta_V)$ such that $(\varphi \times \varphi)(\cR |_U) = \cS |_V$.
\end{itemize}

If $\Gamma \grpaction{} (X, \mu)$ is a pmp action of a discrete countable group, the {\em orbit equivalence relation}  $\cR(\Gamma \grpaction{} X)$ is defined by 
\begin{equation*}
  (x, y) \in \cR(\Gamma \grpaction{} X) \ \text{if and only if there exists} \ g \in \Gamma \ \text{such that} \ y = g \cdot x
  \eqstop
\end{equation*}
Then $\cR(\Gamma \grpaction{} X)$ is ergodic if and only if $\Gamma$ acts ergodically on $(X, \mu)$.

\begin{definition}[\cite{gromov93-asymptotic-invariants, furman99}]
  \label{def:ME}
  Two discrete countable groups $\Gamma$ and $\Lambda$ are called {\em measure equivalent} if there exist free ergodic pmp actions $\Gamma \grpaction{} (X, \mu)$ and $\Lambda \grpaction{} (Y, \eta)$ such that $\cR(\Gamma \grpaction{} X)$ and $\cR(\Lambda \grpaction{} Y)$ are stably orbit equivalent.
\end{definition}

Let $\cS \subset \cR$ be an inclusion of pmp equivalence relations.  We say that $\cS \subset \cR$ has {\em finite index} if almost every $\cR$-class is a finite union of $\cS$-classes. If the inclusion $\cS \subset \cR$ has finite index, then for every non-negligible Borel subset $U \subset X$, the inclusion $\cS |_{U} \subset \cR |_{U}$ has finite index as well. For every $\varphi \in [[\cR]]$, denote by $\dom(\varphi)$ the {\em domain} of $\varphi$ and by $\rng(\varphi)$ the {\em range} of $\varphi$.

\begin{definition}
 \label{def:commensurable}
Let $\cS \subset \cR$ be an inclusion of pmp equivalence relations defined on the standard probability space $(X, \mu)$. The {\em quasi-normaliser} of $\mathcal S$ inside $\mathcal R$, denoted by $  \QN_\cR(\cS)$, is defined as the set of all $\varphi \in [[\cR]]$ such that both of the inclusions
\begin{equation*}
  (\vphi \times \vphi)(\cS |_{\dom(\varphi)}) \cap \cS |_{\rng(\varphi)} \subset \cS |_{\rng(\varphi)} \ \text{and} \ (\vphi^{-1} \times \vphi^{-1})(\cS |_{\rng(\varphi)}) \cap \cS |_{\dom(\varphi)} \subset \cS |_{\dom(\varphi)}
\end{equation*}
have finite index. We then say that $\cS$ is {\em commensurated} by $\cR$ if $\QN_{\cR}(\cS)$ generates the equivalence relation $\cR$.
\end{definition}

If $\Lambda \leq \Gamma$ is a commensurated subgroup of a discrete countable group and $\Gamma \curvearrowright (X, \mu)$ is a pmp action, the subequivalence relation $\mathcal R(\Lambda \curvearrowright X) \subset \mathcal R(\Gamma \curvearrowright X)$ is commensurated.

If the subequivalence relation $\cS \subset \cR$ is commensurated, then for every non-negligible Borel subset $U \subset X$, the subequivalence relation $\cS |_{U} \subset \cR |_{U}$ is commensurated as well, by \cite[Lemma 3.18]{kida11-BS}.

\section{The relative profinite completion of Baumslag-Solitar groups}
\label{sec:products-of-BS}

In this section, we prove several results on the structure of relative profinite completions of Baumslag-Solitar groups and their products. 

\begin{proposition}\label{prop:conjugation-compact-subgroups}
  Let $1 \leq |m| \leq n$ and $H \leq \rG(m, n)$ be any compact subgroup. Then there exists $g \in \rG(m, n)$ such that $g H g^{-1} \subset \rK(m,n)$.
\end{proposition}
\begin{proof}
 Let $T$ be the Bass-Serre tree of $\BS(m,n)$ and $H \leq \rG(m,n)$ a compact subgroup of $\rG(m,n)$.  We claim that $H$ fixes some vertex of $T$.  If this is the case, then we can conjugate $H$ by an element $g \in \rG(m,n)$ such that $g H g^{-1}$ stabilises $\langle a \rangle$, since the action of $\rG(m, n)$ on $V(T)$ is transitive. This will finish the proof.

Let $v \in V(T)$ be any vertex. Since $H$ is compact, the orbit $Hv$ is necessarily finite. Denote by $S \subset T$ the smallest subtree whose set of vertices contains the orbit $Hv$.  Since $Hv$ is finite, $S$ is a finite subtree of $T$.  It is moreover invariant under the action of $H$, since $H$ acts by simplicial automorphisms of $T$.  Denote the boundary of $S$ by $\partial S$, that is the set of vertices in $V(S)$ that have only one neighbour in $S$.  Then $\partial S$ is also invariant under the action of $H$.  It follows that $V(S) \setminus \partial S$ is invariant under $H$ and so is the subtree $S' \subset S$ whose set of vertices is $V(S') = V(S) \setminus \partial S$.  Replacing $S$ by $S'$ repeatedly, we can assume that $S$ is either a single vertex or an edge with its two endpoints.  In the first case we are done.  If $H$ stabilises an edge, then it either contains an edge inversion or it fixes both endpoints of the edge pointwise. Since $\BS(m,n)$ acts without edge inversions on $T$, also $\rG(m,n)$ acts without edge inversions on $T$.  We have shown that $H$ stabilises some vertex of $T$, finishing the proof of the proposition.
\end{proof}

\begin{corollary}
  \label{cor:no-compact-normal-subgroups}
 Let $1 \leq |m| \leq n$. Then $\rG(m,n)$ has no non-trivial compact normal subgroup.
\end{corollary}
\begin{proof}
Denote by $T$ the Bass-Serre tree of $\BS(m,n)$ and let $H \lhd \rG(m,n)$ be a compact normal subgroup. By Proposition~\ref{prop:conjugation-compact-subgroups}, there is $g \in \rG(m,n)$ such that $H =  g H g^{-1} \subset \rK(m, n)$. Hence $H$ fixes $\langle a \rangle \in T$.  Since $H$ is normal in $\rG(m,n)$ and since $\rG(m, n)$ acts transitively on $V(T)$, we obtain that $H$ fixes $V(T)$ pointwise and so $H$ is the trivial group.
\end{proof}

The following lemma will turn out to be useful in the complete classification of relative profinite completions of products of Baumslag-Solitar groups.
\begin{lemma}
  \label{lem:recover-n-from-index}
  Let $1 \leq |m| \leq n$.  Then
  \begin{equation*}
    \min \bigl \{[\rK(m,n) : \rK(m,n) \cap g \rK(m,n) g^{-1}] \amid g \in \rG(m,n) \setminus \rK(m,n) \bigr \} = |m|
    \eqstop
  \end{equation*}
\end{lemma}
\begin{proof}
If $|m| = n$, we have $\rK(m, n) = \langle a \rangle/\langle a^m\rangle = \ZZ/m\ZZ$ and $\rG(m, n) = \BS(m, n) / \langle a^m\rangle \cong \ZZ/m\ZZ \ast \ZZ$. For every $g \in \rG(m, n) \setminus \rK(m, n)$, we have $\rK(m,n) \cap g \rK(m,n) g^{-1} = \{e\}$ and so ${[\rK(m,n) : \rK(m,n) \cap g \rK(m,n) g^{-1}]} = m$.

Now assume that $|m| \neq n$.  We first show that the minimum of $\{[\langle a \rangle : \langle a \rangle \cap g \langle a \rangle g^{-1}] \amid g \in \BS(m,n) \setminus \langle a \rangle\}$ equals $|m|$.  Note that $[\langle a \rangle : \langle a \rangle \cap t^{-1} \langle a \rangle t] = [\langle a \rangle : \langle a^m \rangle] = |m|$.  Take $g \in \BS(m,n) \setminus \langle a \rangle$ an arbitrary freely reduced word without pinches.  We may assume that the uttermost left letter of $g$ is either $t$ or $t^{-1}$. For all $i,j \in \ZZ$, we have $g a^i g^{-1} = a^j$ if and only if $a^i = g^{-1} a^j g$. Since $|m| \leq n$, the word $g^{-1} a^j g$ is freely reduced and has no pinches for every $j \in \{1, \dotsc, |m|-1\}$. It follows that $|\langle a\rangle \cdot (g \langle a\rangle)| \geq m$ and hence $[\langle a \rangle : \langle a \rangle \cap g \langle a \rangle g^{-1}] =  |\langle a\rangle \cdot (g \langle a\rangle)| \geq m$. Therefore, we have
  \begin{equation*}
    \min \{[\langle a \rangle : \langle a \rangle \cap g \langle a \rangle g^{-1}] \amid g \in \BS(m,n) \setminus \langle a \rangle\} = |m|
    \eqstop
  \end{equation*}

Note that since $|m| \neq n$, $\BS(m, n)$ is a dense subgroup of $\rG(m, n)$. Observe also that since the function $\rG(m, n) \ni g \mapsto {[\rK(m,n) : \rK(m,n) \cap g \rK(m,n) g^{-1}]} \in \NN \setminus \{0\}$ is right $\rK(m,n)$-invariant, it is continuous. Moreover, for every $g \in \BS(m,n)$, we have
  \begin{align*}
   [\rK(m,n) : \rK(m,n) \cap g \rK(m,n) g^{-1}] &= |\rK(m,n) \cdot (g \langle a \rangle)| \\
&=  |\langle a\rangle \cdot (g \langle a\rangle)| \\
 &= [\langle a \rangle : \langle a \rangle \cap g \langle a \rangle g^{-1}].
  \end{align*}
Since $\BS(m, n) \setminus \langle a\rangle$ is dense in $\rG(m, n) \setminus \rK(m, n)$, we obtain
  \begin{equation*}
    \min \bigl \{[\rK(m,n) : \rK(m,n) \cap g \rK(m,n) g^{-1}] \amid g \in \rG(m,n) \setminus \rK(m,n) \bigr\} \\
    = |m|
    \eqstop \qedhere
  \end{equation*} 
\end{proof}

We will now derive from Lemma \ref{lem:recover-n-from-index} two useful facts regarding the structure of $\rG(m, n)$.

\begin{proposition} \label{prop:uniqueness-compact-opens-subgroup}
Let $2 \leq |m| \leq n$. Then $\rK(m, n)$ is a maximal compact subgroup of $\rG(m, n)$. Moreover, any maximal compact subgroup of $\rG(m, n)$ is conjugate to $\rK(m, n)$.
\end{proposition}

\begin{proof}
Let $H \leq \rG(m, n)$ be any compact subgroup such that $\rK(m, n) \leq H$. By Proposition \ref{prop:conjugation-compact-subgroups}, there exists $g \in \rG(m, n)$ such that $gHg^{-1} \subset \rK(m, n)$. Then we have $\rK(m, n) \leq H \leq g^{-1} \rK(m, n) g$ implying that  $ \rK(m, n) \cap g^{-1} \rK(m, n) g = \rK(m, n)$. Since $2 \leq |m| \leq n$, we obtain $ g \in \rK(m, n)$ by Lemma~\ref{lem:recover-n-from-index}. This shows that $H = \rK(m, n)$ and thus $\rK(m, n)$ is maximal compact.

Let $H$ be any maximal compact subgroup of $\rG(m, n)$. By Proposition \ref{prop:conjugation-compact-subgroups}, there exists $g \in \rG(m, n)$ such that $g H g^{-1} \leq \rK(m, n)$. Since $H$ is maximal compact, $g H g^{-1}$ is maximal compact as well. Therefore $g H g^{-1} = \rK(m, n)$.
\end{proof}

\begin{proposition}
  \label{prop:trivial-centre}
  Let $2 \leq |m| \leq n$. The centre of $\rG(m,n)$ is trivial.
\end{proposition}
\begin{proof}
Denote by $T$ the Bass-Serre tree of $\BS(m,n)$ and let $g \in \mathcal Z(\rG(m, n))$. Since $g \rK(m, n) g^{-1} = \rK(m, n)$ and $2 \leq |m|$, we have $g \in \rK(m, n)$ by Lemma \ref{lem:recover-n-from-index}. So the centre $\mathcal Z(\rG(m,n))$ is contained in $\rK(m,n)$ showing that it is a compact normal subgroup of $\rG(m,n)$.  By Corollary \ref{cor:no-compact-normal-subgroups}, it follows that $\cZ(\rG(m,n)) = \{e\}$.
\end{proof}

Before proving Theorem \ref{thm:main-result-on-completions}, we start by classifying the relative profinite completions of Baumslag-Solitar groups, up to isomorphism.

\begin{theorem}
  \label{thm:classification-completion-Baumslag-Solitar-groups}
  Let $2 \leq |m| \leq n$ and $2 \leq |p| \leq q$.  Then $\rG(m,n) \cong \rG(p,q)$ if and only if 
  \begin{itemize}
\item  either $|m| = n = |p| = q$
\item or $|m| \neq n$ and $(m,n) = (p,q)$.
\end{itemize}
\end{theorem}
\begin{proof}
Assume first that $|m| = n$. Then we have $\rG(m,n) \cong \ZZ / m\ZZ \ast \ZZ$.  If $|p| \neq q$, then $\rG(p,q)$ is not a discrete group, hence $|p| = q$. The largest finite order of an element in $\ZZ/ m\ZZ * \ZZ$ is $|m|$, so $\ZZ/ m\ZZ * \ZZ \cong \ZZ/ p\ZZ \ast \ZZ$ if and only if $|m| = |p|$. 

Assume now that $|m| \neq n$. Let $d = \mathrm{gcd}(m,n)$, $m_0 = |m|/d$ and $n_0 = n/d$.  Denote by $s : \rG(m,n) \ra \NN \setminus \{0\}$ the scale function of $\rG(m,n)$.  By \cite{elderwillis13}, we have $s(\rG(m,n)) = |m_0|^{\NN} \cup n_0^{\NN}$.  Since $m_0$ and $n_0$ are coprime, the pair $(|m_0|, n_0)$ is an invariant of $\rG(m,n)$.

Recall from Proposition \ref{prop:uniqueness-compact-opens-subgroup}, that $\rK(m,n)$ is, up to conjugacy, the unique maximal compact subgroup of $\rG(m,n)$.  By Lemma \ref{lem:recover-n-from-index}, 
  \begin{equation*}
    \min \{[\rK(m,n) : \rK(m,n) \cap g \rK(m,n) g^{-1}] \amid g \in \rG(m,n) \setminus \rK(m,n)\} = |m|
    \eqstop
  \end{equation*}
  We infer that $|m|$ is an invariant of $\rG(m,n)$. Thus $d = |m|/m_0$ and $n = d n_0$ are invariants of $\rG(m,n)$ as well. 

We are left to prove that $\rG(m, n)$ and $\rG(-m, n)$ are not isomorphic. Assume by contradiction that there exists an isomorphism $\pi : \rG(m, n) \to \rG(-m, n)$. We may assume that $m > 0$. By Proposition \ref{prop:uniqueness-compact-opens-subgroup}, up to conjugating by an inner automorphism, we may assume that $\pi(\rK(m, n)) = \rK(-m, n)$. Consider $\rG(m, n) = \langle \rK(m, n), t\rangle$ and $\rG(-m, n) = \langle \rK(-m, n), u \rangle$ with $\rK(m, n) = \overline{\langle a\rangle}$ and $\rK(-m, n) = \overline{\langle b\rangle}$. Observe that since $\rK(m, n)$ and $\rK(-m, n)$ are both profinite completions of the integer ring $\ZZ$, they are both compact unital rings. Momentarily, put $\mathbb A = \rK(m, n)$ with $1_{\mathbb A} = a$ and $\mathbb B = \rK(-m, n)$ with $1_{\mathbb B} = b$. Take a sequence $n_i \in \ZZ$ such that $\lim_i \pi(n_i 1_{\mathbb A}) = 1_{\mathbb B}$. Since $\mathbb B$ is compact, up to passing to a subsequence, we may assume that $n_i 1_{\mathbb B}$ is convergent in $\mathbb B$. Write $r = \lim_i n_i 1_{\mathbb B}$. We have
\begin{equation*}
  r \pi(1_{\mathbb A}) = \lim_i \, (n_i 1_{\mathbb B}) \, \pi(1_{\mathbb A}) = \lim_i n_i \pi(1_{\mathbb A}) = \lim_i \pi(n_i 1_{\mathbb A}) = 1_{\mathbb B}
\end{equation*}
and so $\pi(1_{\mathbb A})$ is a unit of $\mathbb B$. Next, define the automorphism $\rho : \mathbb B \to \mathbb B$ by $\rho(x) = rx$ for all $x \in \mathbb B$. Since $\rho(-m x) = - m \rho(x)$ and $\rho(n x) = n\rho(x)$ for all $x \in \rK(-m, n)$, $\rho$ extends to an automorphism of $\rG(-m, n)$ by letting $\rho(u) = u$. Note that $\rho$ and $\rho^{-1}$ are continuous on $\rG(-m, n)$ since $\rK(-m, n)$ is an open subgroup of $\rG(-m, n)$ and $\rho$ and $\rho^{-1}$ are continuous on $\rK(-m, n)$. Thus $\rho \circ \pi : \rG(m, n) \to \rG(-m, n)$ is an isomorphism such that $(\rho \circ \pi)(a) = b$. Hence, up to composing $\pi$ with $\rho$,  we may assume that $\pi(a) = b$. 

Write $\Delta_\pm : \rG(\pm m, n) \to \mathbb R_+$ for the modular function. Fix a left-invariant Haar measure $\mu$ on $\rG(-m, n)$ such that $\mu(\rK(-m, n)) = 1$. We then have $\mu(m \rK(-m, n)) = 1/m$ (resp.\ $\mu(n \rK(-m, n)) = 1/n$) since $[\rK(-m, n) : m \rK(-m, n)] = m$ (resp.\ $[\rK(-m, n) : n \rK(-m, n)] = n$). We moreover have
$$\frac1n = \mu(n \rK(-m, n)) = \mu(u (m \rK(-m, n)) u^{-1}) = \Delta_{-}(u^{-1}) \mu(m \rK(-m, n)) = \Delta_{-}(u^{-1}) \frac1m$$
and so $\Delta_{-}(u) = n/m$. Likewise, we have $\Delta_+(t) = n/m$. Then we get
\begin{equation*}
  \Delta_{-}(\pi(t)^{-1} u)
  =
  \frac{1}{\Delta_{-}(\pi(t))} \Delta_{-}(u)
  =
  \frac{1}{\Delta_+(t)} \Delta_{-}(u)
  =
  \frac{m}{n} \frac{n}{m}
  =
  1
  \eqstop
\end{equation*}
Recall from Proposition \ref{prop:relative-profinite-completion-of-HNN} that $\rG(-m,n)$ is an HNN-extension of $\rK(-m,n)$ whose free letter is $u$.  Since $\Delta_{-}(\pi(t)^{-1}u) = 1$ and $m \neq n$, any word representing $\pi(t)^{-1}u$ in $\rG(-m,n)$ contains as many $u$'s as $u^{-1}$'s.  It follows that if $k$ denotes the number of $u$'s appearing in a word representing $\pi(t)^{-1}u$, then $m^k n^k \rK(-m,n)$ commutes with $\pi(t)^{-1}u$.  We may assume that $k \geq 1$. We get 
\begin{equation*}
  b^{-m^{k - 1} n^{k + 1}} = u b^{m^k n^k} u^{-1} = \pi(t) b^{m^k n^k} \pi(t)^{-1} = \pi(t a^{m^{k} n^{k}} t^{-1}) = \pi(a^{m^{k - 1}n^{k + 1}}) = b^{m^{k - 1} n^{k + 1}},
\end{equation*}
which is a contradiction. Therefore, $\rG(m, n)$ and $\rG(-m, n)$ are not isomorphic. This finishes the proof.
\end{proof}

\begin{proof}[Proof of Theorem \ref{thm:main-result-on-completions}]
 
We use the following notation throughout the proof.
  \begin{itemize}
    \item $G_i = \rG(m_i, n_i)$, $K_i = \rK(m_i, n_i)$, $H_j = \rG(p_j, q_j)$ and $L_j = \rK(p_j, q_j)$ for all  $i \in \{1, \dotsc, k\}$ and all $j \in \{1, \dots, l\}$.  
    \item $G = G_1 \times \dotsm \times G_k$ and $K = K_1 \times \dotsm \times K_k$.
    \item $H = H_1 \times \dotsm \times H_l$ and $L = L_1 \times \dotsm \times L_l$.
  \end{itemize}

We prove the result by induction on $\max(k, l)$. If $\max(k, l) = 1$, then we necessarily have $k = l = 1$ and the result follows from Theorem \ref{thm:classification-completion-Baumslag-Solitar-groups} in that case. We assume now that $\max(k, l) \geq 2$. Let $\pi : G \to H$ be an isomorphism. Since by Proposition \ref{prop:uniqueness-compact-opens-subgroup} we know that $K_i$ is maximal compact in $G_i$ for every $i \in \{1, \dotsc, k\}$, also $K$ is maximal compact in $G$. Similarly $L$ is the unique maximal compact subgroup of $H$, up to conjugacy. Therefore, we can compose $\pi$ with an inner automorphism of $H$ in order to assume that $\pi(K) = L$.  

Possibly considering the inverse of $\pi$ and permuting the factors of $G = G_1 \times \dotsm \times G_k$, we may assume that $|m_1| = \min ( \{ |m_i|\}_{i \leq k} \cup  \{|p_j|\}_{j \leq l})$.  Let $g = (t_1^{-1}, e, \dotsc, e) \in G$, where $t_1$ denotes the free letter of $\BS(m_1,n_1)$.  Since $[K_1 : K_1 \cap t_1^{-1} K_1 t_1] = |m_1|$, we obtain 
  \begin{equation*}
    [L : L \cap \pi(g) L \pi(g)^{-1}]
    =
    [K : K \cap g K g^{-1}]
    =
    |m_1|
    \eqstop
  \end{equation*}
  Writing $\pi(g) = (h_1, \dotsc, h_l) \in H_1 \times \dotsm \times H_l$, we also have
  \begin{equation*}
    [L : L \cap \pi(g) L \pi(g)^{-1}]
    =
    \prod_{j \in \{1, \dotsc, l\}} [L_j : L_j \cap h_j L_j h_j^{-1}]
    \eqstop
  \end{equation*}
  By Lemma \ref{lem:recover-n-from-index}, the minimum of $\{[L_j : L_j \cap h L_j h^{-1}] \amid h \in H_j \setminus L_j\}$ equals $|p_j|$.  Combining this with the fact that $|m_1| = \min (\{ |m_i|\}_{i \leq k} \cup  \{|p_i|\}_{j \leq l})$, we obtain that there is some $j_0 \in \{1, \dotsc, l\}$ such that $[L_{j_0} : L_{j_0} \cap h_{j_0} L_{j_0} h_{j_0}^{-1}] = |m_1| = |p_{j_0}|$ and $h_j \in L_j$ for all $j \neq j_0$.  Permuting the factors of $L$, we can assume that $j_0 = 1$.  We have proved that $\pi(G_1 \times K_2 \times \dotsm \times K_k) \leq H_1 \times L_2 \times \dotsm \times L_l$.

  Since $|p_1| = |m_1| = \min (\{ |m_i|\}_{i \leq k} \cup  \{|p_j|\}_{j \leq l})$, the same argument as before implies that there is $i_0 \in \{1, \dotsc, k\}$ such that $\pi^{-1}(H_1 \times L_2 \times \dotsm \times L_l) \leq K_1 \times \dotsm \times G_{i_0} \times \dotsm \times K_k$.  Since
  \begin{align*}
    G_1 \times K_2 \times \dotsm \times K_k &=
    \pi^{-1}(\pi(G_1 \times K_2 \times \dotsm \times K_k)) \\
  &  \leq \pi^{-1}(H_1 \times L_2 \times \dotsm \times L_l) \\
  & \leq K_1 \times \dotsm \times G_{i_0} \times \dotsm \times K_k
    \eqcomma
  \end{align*}
  it follows that $i_0 = 1$.  So $\pi(G_1 \times K_2 \times \dotsm \times K_k) = H_1 \times L_2 \times \dotsm \times L_l$.

  Take $g \in G_1 \setminus \{e\}$ an arbitrary non-trivial element and write $\pi(g,e, \dotsc, e) = (h_1, \dotsc, h_l) \in {H_1 \times L_2 \times \dotsm \times L_l}$.  Then
  \begin{align*}
    \pi(\{e\} \times G_2 \times \dotsm \times G_k)
    & \subset
    \pi(\cZ_G( g, e, \dots, e)) \\
    & =
    \cZ_H(\pi(g, e, \dots, e)) \\
    & = 
    \cZ_H(h_1, \dots, h_j, \dots, h_l) \\
    & 
    \subset H_1 \times \cdots \times H_{j - 1} \times \mathcal Z_{H_j}(h_j) \times H_{j + 1} \times \cdots \times H_l. 
  \end{align*}
 Since $h_j \in L_j$ and $L_j$ is abelian, we have $L_j \subset \mathcal Z_{H_j}(h_j)$, hence
 \begin{align*}
\pi(G_1 \times \{e\} \times \cdots \times \{e\}) &\subset H_1 \times L_2 \times \cdots \times L_l \\
& \subset  H_1 \times \cdots \times H_{j - 1} \times \mathcal Z_{H_j}(h_j) \times H_{j + 1} \times \cdots \times H_l. 
\end{align*}
 We get $\pi(G) = H \subset H_1 \times \cdots \times H_{j - 1} \times \mathcal Z_{H_j}(h_j) \times H_{j + 1} \times \cdots \times H_l$ and hence $\mathcal Z_{H_j}(h_j) = H_j$. By Proposition \ref{prop:trivial-centre}, we have $h_j = e$. So, $\pi(G_1 \times \{e\} \times \dotsm \times \{e\}) \subset H_1 \times \{e\} \times \dotsm \times \{e\}$. Likewise, we have $\pi^{-1}(H_1 \times \{e\} \times \dotsm \times \{e\}) \subset G_1 \times \{e\} \times \dotsm \times \{e\}$, hence 
 \begin{equation*}
   \pi(G_1 \times \{e\} \times \dotsm \times \{e\}) = H_1 \times \{e\} \times \dotsm \times \{e\}
   \eqstop
 \end{equation*}
 In particular $G_1 \cong H_1$. By Theorem \ref{thm:classification-completion-Baumslag-Solitar-groups}, we have $|m_1| = n_1 = |p_1| = q_1$ or $(m_1, n_1) = (p_1, q_1)$. Applying again Proposition \ref{prop:trivial-centre}, we get
  \begin{equation*}
    \pi(\{e\} \times G_2 \times \dotsm \times G_k) = \pi(\cZ_G(G_1 \times \{e\} \times \dotsm \times \{e\})) = \cZ_H(H_1 \times \{e\} \times \dotsm \times \{e\}) = \{e\} \times H_2 \times \dotsm \times H_k
    \eqstop
  \end{equation*}
 It follows that $G_2 \times \dotsm \times G_k \cong H_2 \times \dotsm \times H_l$. Since $\max(k - 1, l - 1) = \max(k, l) - 1$, the induction hypothesis yields $k - 1 = l - 1$, that is, $k = l$ and there exists a permutation $\sigma \in \Sym(k - 1)$ such that $G_i \cong H_{\sigma(i)}$ for every $i \in \{2, \dots, k\}$. Finally, an application of Theorem \ref{thm:classification-completion-Baumslag-Solitar-groups} finishes the proof of Theorem  \ref{thm:main-result-on-completions}.
 \end{proof}

\section{An algebraic invariant for aperiodic subequivalence relations}
\label{sec:invariant}

In this section, $\cS \subset \cR$ denotes an inclusion of pmp equivalence relations defined on a standard probability space $(X, \mu)$.  We construct an invariant of so called {\em aperiodic} inclusions based on index cocycles.  This requires some preliminary work on invariants for cocycles.

\subsection{Ergodic 1-cocycles}
\label{sec:ergodic-cocycles}

Let $\mathcal R$ be a pmp equivalence relation defined on the standard probability space $(X, \mu)$ and $G$ a topological group. A $1$-{\em cocycle} $\Omega$ for $\mathcal R$ with values in $G$ is a Borel map $\Omega : \mathcal R \to G$ such that 
\begin{equation}\label{cocycle}
\Omega(x, z) = \Omega(x, y) \,  \Omega(y, z) \ \text{for almost every} \ (x, y, z) \in \mathcal R^{(2)}.
\end{equation} 
Here we mean for $\nu^{(2)}$-almost every $(x, y, z) \in \mathcal R^{(2)}$, where $\nu^{(2)}$ is the measure constructed in Section \ref{sec:pmp-equivalence-relations}. Two cocycles $\Omega_1, \Omega_2 : \mathcal R \to G$ are {\em cohomologous} if there exists a Borel map $c : X \to G$ such that 
\begin{equation*}
  \Omega_2(x, y) = c(x) \, \Omega_1(x, y) \, c(y)^{-1} \ \text{for almost every} \ (x, y) \in \mathcal R
  \eqstop
\end{equation*}
Here we mean for $\nu$-almost every $(x, y) \in \mathcal R$, where $\nu$ is the measure constructed in Section \ref{sec:pmp-equivalence-relations}.

In \cite{kaimanovichschmidt}, Kaimanovich and Schmidt introduced the notion of ergodic 1-cocycles.  We use their definition in the setting of cocycles with values in non-Archimedian Polish groups. Recall that a Polish group $G$ is {\em non-Archimedian} if it admits a neighbourhood basis $(K_n)_{n \in \NN}$ of the neutral element consisting of open subgroups of $G$.
  
\begin{definition}[See {\cite[Definition 2.3]{kaimanovichschmidt}}]
  \label{def:approximately-ergodic-cocycles}
 Let $\mathcal R$ be a pmp equivalence relation defined on $(X, \mu)$ and $G$ a non-Archimedian Polish group. Let ${\Omega: \cR \ra G}$ be a 1-cocycle. We say that $\Omega$ is {\em ergodic} for $\mathcal R$ if for every open subgroup $K \leq G$ the subequivalence relation $\{(x,y) \in \cR \amid \Omega(x,y) \in K\}$ has the same invariant Borel sets as $\cR$.
\end{definition}

\begin{remark}
  As already pointed out in \cite{kaimanovichschmidt}, the notion of ergodic cocycles is not invariant under passing to cohomologous cocycles.
\end{remark}

\begin{proposition}[See {\cite[Proposition 2.5]{kaimanovichschmidt}}]
  \label{prop:ergodic-cocycles-give-invariant}
  Let $\cR$ be an ergodic pmp equivalence relation defined on $(X, \mu)$ and $\Omega: \cR \ra G$ an ergodic 1-cocycle with values in a non-Archimedian Polish group. Then there are a unique minimal closed subgroup $H \leq G$ and a co-negligible $\mathcal R$-invariant Borel subset $X_0 \subset X$ such that $\Omega(\cR|_{X_0}) \subset H$ and for every non-negligible Borel subset $U \subset X_0$, we have that $\Omega(\cR |_U)$ is dense in $H$.
\end{proposition}
\begin{proof}
  We follow the proof of \cite[Proposition 2.5]{kaimanovichschmidt}.  Let $(K_n)_{n \in \NN}$ be a neighbourhood basis of $e$ consisting of open subgroups of $G$. After discarding a co-negligible $\mathcal R$-invariant Borel subset of $X$, we may assume that $\Omega$ is a {\em strict} $1$-cocycle, that is, a $1$-cocycle for which the cocycle identity (\ref{cocycle}) holds {\em everywhere}. For every $n \in \NN$, the equivalence relation $\cR_n = \{ (x,y) \in \cR \amid \Omega(x, y) \in K_n\}$ is ergodic by assumption.  For every non-negligible Borel subset $A \subset X$  and every $n \in \NN$, define
  \begin{gather*}
    H_n(A)
    =
    \{g \in G \amid \exists \vphi \in [[\cR|_A]]: \mu(\dom(\varphi)) > 0 \ \text{and} \ \Omega(x, \vphi(x)) \in g K_n \text{ for all } x \in \dom(\vphi)\}
    \subset G
    \eqstop
  \end{gather*}
  Put $H(A) = \bigcap_{n \in \NN} H_n(A)$. Observe that $H(A) \subset G$ is the essential image of the restriction $\Omega_A : \mathcal R|_A \to G$, that is, $H(A)$ is the smallest closed subset $F \subset G$ such that $\nu(\Omega_A^{-1}(G \setminus F)) = 0$.
  
  Let $A, B \subset X$ be non-negligible.  We show that $H(A) = H(B)$.  Take $n \in \NN$.  For $g \in H_n(A)$, there is a partial isomorphism $\vphi \in [[\cR|_A]]$ such that $\mu(\dom(\varphi)) > 0$ and $\Omega(x, \vphi(x)) \in g K_n$ for all $x \in \dom(\vphi)$.  Take $m \in \NN$ large enough such that $g^{-1} K_m g, K_m \subset K_n$.  Since $\cR_m$ is ergodic, there exist a non-negligible Borel subset $D \subset B$ and partial isomorphisms $\psi_1, \psi_2 \in [[\cR_m]]$ defined on $D$ such that $\psi_1(D) \subset \dom(\vphi)$ and $\psi_2(D) = (\vphi \circ \psi_1)(D)$.  It follows that for almost every $x \in D$, we have $(\psi_2^{-1} \circ \vphi \circ \psi_1)(x) \in B$ and
  \begin{align*}
   \Omega(x, (\psi_2^{-1} \circ \vphi \circ \psi_1)(x)) &=
     \Omega(x, \psi_1(x)) \, \Omega(\psi_1(x), (\vphi \circ \psi_1)(x)) \, \Omega((\vphi \circ \psi_1)(x),(\psi_2^{-1} \circ \vphi \circ \psi_1)(x)) \\
    & \in K_m g K_n K_m \\
    & = g (g^{-1} K_m g)K_n K_m \\
    & \subset g K_n
    \eqstop
  \end{align*}
  This shows that $g \in H_n(B)$.  By symmetry, it follows that $H_n(A) = H_n(B)$ for all $n \in \NN$ and hence $H(A) = H(B)$.

Put $H = H(X)$.  We show that $H \subset G$ is a subgroup.  Take $g_1, g_2 \in H$, $n \in \NN$ and take $\vphi \in [[\cR ]]$ such that $\mu(\dom(\varphi)) > 0$ and $\Omega(x, \vphi(x)) \in g_2 K_n$ for all $x \in \dom(\vphi)$.  Let $m \in \NN$ be large enough such that $g_2^{-1}g_1 K_m g_1^{-1}g_2 \subset K_n$.  Since $H(\dom(\vphi)) = H$,  there is some $\psi \in [[\cR|_{\dom(\vphi)}]]$ such that $\mu(\dom(\psi)) > 0$ and $\Omega(y, \psi^{-1}(y)) \in g_1 K_m$ for all $y \in \dom(\psi^{-1})$.  It follows that for all $x \in \dom(\psi)$, 
  \begin{align*}
    \Omega(x, (\vphi \circ \psi)(x))
    & =
    \Omega(x, \psi(x)) \, \Omega(\psi(x), (\vphi \circ \psi)(x)) \\
    & \in
    K_m g_1^{-1} g_2 K_n \\
    & =
    g_1^{-1} g_2 (g_2^{-1}g_1 K_m g_1^{-1}g_2) K_n \\
    & \subset
    g_1^{-1} g_2 K_n
    \eqstop
  \end{align*}
  We have shown that $g_1^{-1}g_2 \in H$,  hence $H$ is a group.

  
  Since $H = H(X)$ is the essential image of $\Omega : \mathcal R \to G$, we have $\nu(\Omega^{-1}(G \setminus H)) = 0$, that is, $\Omega(x,y) \in H$ for almost every $(x, y) \in \cR$. Hence there exists a co-negligible $\mathcal R$-invariant Borel subset subset $X_0 \subset X$ such that
  \begin{itemize}
  \item $\Omega(\mathcal R |_{X_0}) \subset H$ and
  \item $\Omega$ is a strict cocycle on $\cR|_{X_0}$.
  \end{itemize}
 Let $U \subset X_0$ be a non-negligible Borel subset. We have $\Omega(\mathcal R |_U) \subset H$ and since $H = H(U)$ is the essential image of the restriction $\Omega_U : \mathcal R |_U \to G$, it follows that $\overline{\Omega(\mathcal R |_U)} = H$. This finishes the proof.
\end{proof}

We will use the notation $H(\Omega) = H(X)$ in order to refer to the group $H$ defined by the ergodic 1-cocycle $\Omega$.  Moreover, $H(\Omega)(U) = H(U)$ as introduced in the previous proof is used for the rest of this section.  We already saw that for an ergodic cocycle $\Omega$, the group $H(\Omega)(U)$ does not depend on the choice of the non-negligible Borel subset $U \subset X$.  The next proposition gives a strengthening of this result, showing also that $H(\Omega)(U)$ does only depend on the cohomology class of $\Omega$ as long as we assume ergodicity.  

\begin{proposition}
  \label{prop:independence-of-subset-and-representant}
    Let $\cR$ be an ergodic pmp equivalence relation defined on $(X, \mu)$ and ${\Omega_1, \Omega_2: \cR \ra G}$ be cohomologous $1$-cocycles with values in a non-Archimedian Polish group.  Assume that for all $i \in \{1, 2\}$ there is a non-negligible Borel subset $U_i \subset X$ such that $\Omega_i$ is ergodic for $\cR|_{U_i}$. Then $H(\Omega_1)(U_1)$ and $H(\Omega_2)(U_2)$ are conjugate in $G$.
\end{proposition}

\begin{proof}
  Since $\cR$ is ergodic, there are non-negligible subsets $V_i \subset U_i$, $i \in \{1, 2\}$, and a partial isomorphism ${\psi:V_1 \ra V_2}$ in $[[\cR]]$.  Since $\Omega_1$ and $\Omega_2$ are cohomologous, there is a Borel map ${c:X \ra G}$ such that $\Omega_1(x,y) = c(x) \, \Omega_2(x,y) \, c(y)^{-1}$ for almost every $(x,y) \in \cR$.  Define the $1$-cocycle $\wt{\Omega_2} : \cR|_{V_1} \to G$ by the formula $\wt{\Omega_2}(x,y) = \Omega_2(\psi(x), \psi(y))$.  Since $\Omega_2$ is ergodic for $\mathcal R |_{V_2}$, it follows that $\wt{\Omega_2}$ is ergodic for $\mathcal R |_{V_1}$.  We also have $H(\wt{\Omega_2})(V_1) = H(\Omega_2)(V_2) = H(\Omega_2)(U_2)$, where the second equality follows from Proposition \ref{prop:ergodic-cocycles-give-invariant}.  Since moreover $\psi \in [[\cR]]$, we obtain
  \begin{align*}
    \wt{\Omega_2}(x,y)
    & =
    \Omega_2(\psi(x), \psi(y)) \\
    & =
    \Omega_2(\psi(x), x) \, \Omega_2(x,y) \, \Omega_2(y, \psi(y)) \\
    & =
    \Omega_2(\psi(x), x) \, c(x)^{-1} \, \Omega_1(x,y) \, c(y) \, \Omega_2(y, \psi(y))
  \end{align*}
  for almost every $(x, y) \in \mathcal R |_{V_1}$. Hence $\wt{\Omega_2}$ is cohomologous to ${\Omega_1} |_{\mathcal R |_{V_1}}$.  We can then replace $\Omega_2$ by $\wt{\Omega_2}$ and $X$ by $V_1$ in order to assume that $\Omega_1$ and $\Omega_2$ are cohomologous ergodic cocycles for $\cR$.

  Let $c: X \ra G$ be a Borel map such that $\Omega_1(x,y) = c(x) \, \Omega_2(x,y) \, c(y)^{-1}$ for almost every $(x,y) \in \cR$.  Let $g_0 \in G$ be an essential value of $c$.  We show that $g_0 H(\Omega_2) g_0^{-1} \subset H(\Omega_1)$.  If this is proven, the proposition follows by symmetry.

  Let $(K_n)_{n \in \NN}$ be a basis of neighbourhoods of $e$ in $G$ consisting of open subgroups.  For $i \in \{1, 2\}$, let 
  \begin{equation*}
    H_n(\Omega_i)
    =
    \{ g \in G \amid \exists \vphi \in [[\cR]]: \mu(\dom(\varphi)) > 0 \ \text{and} \ \Omega_i(x, \vphi(x)) \in g K_n \text{ for all } x \in \dom(\vphi)\}
  \end{equation*}
  be defined as in the proof of Proposition \ref{prop:ergodic-cocycles-give-invariant}.  Take $g \in H(\Omega_2)$.  We show that for all $n \in \NN$, we have $g_0 g g_0^{-1} \in H_n(\Omega_1)$.  So take $n \in \NN$ arbitrarily and $m \in \NN$ large enough such that
  \begin{gather*}
    ((g_0gg_0^{-1})^{-1} K_m g_0 g g_0^{-1}) (g_0 K_m g_0^{-1}) K_m \subset K_n
    \eqstop
  \end{gather*}
  Since $g_0$ is an essential value of $c$, the set $A = c^{-1}(K_m g_0)$ has positive measure.  Since $H(\Omega_2) = H(\Omega_2)(A)$, there is some $\vphi \in [[\cR|_A]]$ such that $\mu(\dom(\varphi)) > 0$ and $\Omega_2(x,\vphi(x)) \in g K_m$ for all $x \in \mathrm{dom}(\vphi)$.  Then for almost every $x \in \dom(\varphi)$, we have
  \begin{align*}
    \Omega_1(x,\vphi(x))
    & = c(x) \, \Omega_2(x, \vphi(x)) \, c(\vphi(x))^{-1} \\
    & \in K_m g_0 g K_m g_0^{-1} K_m \\
    & = (g_0gg_0^{-1}) ((g_0gg_0^{-1})^{-1} K_m g_0 g g_0^{-1}) (g_0 K_m g_0^{-1}) K_m \\
    & \subset g_0 g g_0^{-1} K_n
    \eqstop
  \end{align*}
  This shows that $g_0gg_0^{-1} \in H_n(\Omega_1)$.  Since $H(\Omega_1) = \bigcap_{n \in \NN} H_n(\Omega_1)$, we have finished the proof.  
\end{proof}

\subsection{Index cocycles}
\label{sec:index-cocycles}

In this section, we apply the concept of ergodicity to index cocycles in order to find stable orbit equivalence invariants for certain inclusions of ergodic pmp equivalence relations.  The concept of index cocycles was first introduced by Feldman, Sutherland and Zimmer in \cite{feldmansutherlandzimmer89-subrelations}. We recall the following well-known fact. For a proof, see \cite[Lemmas 1.1, 1.2 and 1.3]{feldmansutherlandzimmer89-subrelations}.
\begin{proposition}[\cite{feldmansutherlandzimmer89-subrelations}]
  \label{prop:existence-index-cocycle}
  Assume that $\cS \subset \cR$ is an inclusion of ergodic pmp equivalence relations defined on $(X, \mu)$.  There is a countable family $(c_i)_{i \in I}$ of choice functions in $[\cR]$ such that $\id_X \in \{c_i\}_{i \in I}$ and $[x]_\cR = \bigsqcup_{i \in I} [c_i(x)]_\cS$ for almost every $x \in X$.  The map $\Omega: \cR \ra \Sym(I)$ defined by
  \begin{gather*}
    \Omega(x,y)(i) = j \quad \Leftrightarrow \quad [c_j(x)]_\cS = [c_i(y)]_\cS
  \end{gather*}
  is a 1-cocycle.  The cohomology class of $\Omega$ does not depend on the choice of the family $(c_i)_{i \in I}$.
\end{proposition}

Such a $1$-cocycle $\Omega : \mathcal R \to \Sym(I)$ given by Proposition \ref{prop:existence-index-cocycle} will be called an {\em index cocycle}. The next definition makes sense thanks to Propositions \ref{prop:independence-of-subset-and-representant} and \ref{prop:existence-index-cocycle}.

\begin{definition}
  \label{def:aperiodic-inclusions}
    An inclusion $\cS \subset \cR$ of ergodic pmp equivalence relations on $(X, \mu)$ is called {\em aperiodic} if it admits an ergodic index cocycle.
    If $\cS \subset \cR$ is an aperiodic inclusion of pmp equivalence relations, then we define $H(\cS \subset \cR) = H(\Omega)$, for some (or any) ergodic index cocycle $\Omega$ of $\cS \subset \cR$.
\end{definition}

In the next proposition, we show that $H(\cS \subset \cR)$ is a stable orbit equivalent invariant of the inclusion $\cS \subset \cR$.
\begin{proposition}
  \label{prop:aperiodicity-stable}
  Let $\cS \subset \cR$ be an aperiodic inclusion of ergodic pmp equivalence relations on $(X, \mu)$ with an ergodic index cocycle $\Omega : \mathcal R \to \Sym(I)$.  
  \begin{enumerate}
\item Let $\mathcal U \subset \mathcal T$ be an inclusion of ergodic pmp equivalence relations defined on $(Y, \eta)$ and $\varphi : (X, \mu) \to (Y, \eta)$ a pmp Borel automorphism such that $(\varphi \times \varphi)(\cS) = \mathcal U$ and $(\varphi \times \varphi)(\cR) = \mathcal T$. Then  the map $\Omega \circ (\varphi^{-1} \times \varphi^{-1}) : \mathcal T \to \Sym(I)$ is an ergodic index cocycle for the inclusion $\mathcal U \subset \mathcal T$. Therefore the inclusion $\mathcal U \subset \mathcal T$ is aperiodic and moreover we have $H(\cS \subset \cR) \cong H(\mathcal U \subset \mathcal T)$.
\item  Let $U \subset X$ be a non-negligible subset. Then the restriction $\Omega_U = \Omega |_{\mathcal R |_U} : \cR|_U \to \Sym(I)$ is an ergodic index cocycle for the inclusion $\cS|_U \subset \cR|_U$. Therefore the inclusion $\cS|_U \subset \cR|_U$  is aperiodic and moreover we have ${H(\cS|_U \subset \cR|_U)} \cong {H(\cS \subset \cR)}$.
\end{enumerate}
Therefore $H(\cS \subset \cR)$ is a stable orbit equivalent invariant of the inclusion $\cS \subset \cR$.
\end{proposition}
\begin{proof}
  By Proposition \ref{prop:existence-index-cocycle}, there is a family $(c_i)_{i \in I}$ of maps in $[\cR]$ such that $\id_X \in \{c_i\}_{i \in I}$, $[x]_\cR = \bigsqcup_{i \in I} [c_i(x)]_\cS$ for almost every $x \in X$, and the associated index cocycle $\Omega : \mathcal R \to \Sym(I)$ is ergodic. 
  
$(i)$ We have that the map $\Omega \circ (\varphi^{-1} \times \varphi^{-1}) : \mathcal T \to \Sym(I)$ is an ergodic index cocycle for the inclusion $\mathcal U \subset \mathcal T$ associated with the choice functions $\varphi \circ c_i \circ \varphi^{-1} \in [\mathcal T]$, $i \in I$. It follows that the inclusion $\mathcal U \subset \mathcal T$ is aperiodic and moreover we have $H(\cS \subset \cR) \cong H(\mathcal U \subset \mathcal T)$.

$(ii)$ Since $\cR$ is pmp and $\cS$ is ergodic, for every $i \in I$, there is a partial isomorphism $\vphi_i: c_i(U) \ra U$ in $[[\cS]]$.  Let $d_i = (\vphi_i \circ c_i)|_U$.  Then for every $i,j \in I$ and almost every $(x,y) \in \mathcal R |_U$, we have $d_i(x) \sim_\cS d_j(y)$ if and only if $c_i(x) \sim_\cS c_j(y)$. If we denote by $\wt{\Omega} : \mathcal R |_U \to \Sym(I)$ the index cocycle for the inclusion $\cS|_U \subset \cR|_U$ associated with $(d_i)_{i \in I}$, we have $\wt{\Omega}(x, y) = \Omega(x, y)$ for almost every $(x, y) \in \cR |_U$. This shows that $\Omega_U = \Omega |_{\mathcal R |_U} : \cR|_U \to \Sym(I)$ is an index cocycle for the inclusion $\cS|_U \subset \cR|_U$. Moreover, for every open subgroup $K \leq \Sym(I)$, we have
  \begin{equation*}
    \{(x, y) \in \mathcal R |_U \amid \Omega_U(x, y) \in K\} = \{(x, y) \in \mathcal R \amid \Omega(x, y) \in K\} |_U
    \eqstop
  \end{equation*}
  Since $\Omega$ is ergodic for $\mathcal R$, we get that $\Omega_U$ is ergodic for $\cR |_U$.  
  It also follows from Proposition \ref{prop:ergodic-cocycles-give-invariant} that
  \begin{equation*}
    H(\cS|_U \subset \cR|_U) \cong H(\Omega_U) = H(\Omega)(U) \cong H(\cS \subset \cR)
    \eqstop
  \end{equation*}
  This finishes the proof.
\end{proof}

\subsection{Calculation of $H(\cS \subset \cR)$ arising from actions of discrete Hecke pairs}
\label{sec:calculation-of-quotients}

In the following proposition, we calculate the natural index cocycle associated with an inclusion of discrete countable groups acting by pmp transformations on a standard probability space. Recall that for an inclusion of discrete countable groups $\Lambda \leq \Gamma$, the homomorphism $\tau_{\Gamma, \Lambda} : \Gamma \to \Sym(\Gamma / \Lambda)$ is given by left multiplication of $\Gamma$ on the countable  set $\Gamma / \Lambda$.
\begin{proposition}
  \label{prop:commensurable-subgroups-and-equivalence-relations}
  Let $\Lambda \leq \Gamma$ an inclusion of discrete countable groups and let $\Gamma \grpaction{} (X, \mu)$ be a free pmp action such that $\Lambda \grpaction{} (X, \mu)$ is ergodic.  One index cocycle of the inclusion ${\cR(\Lambda \grpaction{} X)} \subset {\cR(\Gamma \grpaction{} X)}$ is given by ${\Omega: \cR(\Gamma \grpaction{} X)} \ra \Sym(\Gamma / \Lambda)$ with $\Omega(g \cdot x, x)= \tau_{\Gamma, \Lambda}(g)$ for every $g \in \Gamma$ and almost every $x \in X$.
  \end{proposition}
\begin{proof}
  Take representatives $s(g \Lambda )_{g \Lambda \in \Gamma / \Lambda}$ of right $\Lambda$-cosets and let the Borel maps $c_{g \Lambda}: X \ra X$, $g \Lambda \in \Gamma / \Lambda$, be defined by $c_{g \Lambda}(x) = s(g \Lambda)^{-1} \cdot x$.  Then we have 
  \begin{gather*}
    [x]_{\cR(\Gamma \grpaction{} X)}
    =
    \Gamma \cdot x
    =
    \bigsqcup_{\Lambda g^{-1} \in \Lambda \backslash \Gamma} \Lambda g^{-1} \cdot x
    =
    \bigsqcup_{g \Lambda \in \Gamma / \Lambda} [c_{g \Lambda} (x)]_{\cR(\Lambda \grpaction{} X)}
    \eqstop
  \end{gather*}
  The index cocycle associated with $(c_{g \Lambda})_{g \Lambda \in \Gamma / \Lambda}$ is given by
  \begin{gather*}
    \Omega(g \cdot x, x)(g_1 \Lambda) = g_2 \Lambda
    \quad \Leftrightarrow \quad
    [c_{g_1 \Lambda}(x)]_{\cR(\Lambda \grpaction{} X)} = [c_{g_2 \Lambda}(g \cdot x)]_{\cR(\Lambda \grpaction{} X)}
    \quad \Leftrightarrow \quad
    \Lambda g_1^{-1} \cdot x = \Lambda g_2^{-1} g \cdot x
    \eqstop
  \end{gather*}
  Since $\Gamma \grpaction{} (X, \mu)$ is free, we obtain that
  \begin{equation*}
    \Omega(g \cdot x, x)(g_1 \Lambda) = g_2 \Lambda
    \quad \Leftrightarrow \quad
    \Lambda g_1^{-1} = \Lambda g_2^{-1} g
    \quad \Leftrightarrow \quad
    g g_1 \Lambda = g_2 \Lambda
    \eqstop
  \end{equation*} 
 Therefore, $\Omega(g \cdot x, x)= \tau_{\Gamma, \Lambda}(g)$ for every $g \in \Gamma$ and almost every $x \in X$.
\end{proof}

Actions of discrete Hecke pairs give examples for which the invariant $H(\cS \subset \cR)$ can be calculated.  This is shown in the following proposition.
\begin{proposition}
  \label{prop:quotients-from-actions}
  Let $\Lambda \leq \Gamma$ be a commensurated subgroup and $\Gamma \grpaction{} (X, \mu)$  a free pmp action such that the action $\Lambda \curvearrowright (X, \mu)$ is aperiodic.  Then the inclusion $\cR(\Lambda \grpaction{} X) \subset \cR(\Gamma \grpaction{} X)$ is aperiodic and $H(\cR(\Lambda \grpaction{} X) \subset \cR(\Gamma \grpaction{} X)) \cong \Gamma \rp \Lambda$.
\end{proposition}
\begin{proof}
  By Proposition \ref{prop:commensurable-subgroups-and-equivalence-relations}, choose the index cocycle $\Omega: \cR(\Gamma \grpaction{} X) \ra \Sym(\Lambda \backslash \Gamma)$ defined by $\Omega(g \cdot x, x)= \tau_{\Gamma, \Lambda}(g)$ for every $g \in \Gamma$ and almost every $x \in X$.  Let $g_1, g_2, \dotsc $ be an enumeration of representatives for $\Gamma / \Lambda$ with $g_1 = e$.  Consider the basis of neighbourhoods of $\id_{\Gamma/\Lambda}$ in $\Sym(\Gamma / \Lambda)$ given by $K_n = {\{\sigma \in \Sym(\Gamma/\Lambda) \amid \sigma(g_i \Lambda) = g_i \Lambda \text{ for all } 1 \leq i \leq n\}}$, $n \geq 1$.  We have
  \begin{align*}
    \cR_n
    & = \{ (x,y) \in \mathcal R(\Gamma \curvearrowright X) \amid \Omega(x,y) \in K_n \} \\
    & = \{(g \cdot x, x)  \in \mathcal R(\Gamma \curvearrowright X) \amid g \in g_i \Lambda g_i^{-1} \text{ for all } 1 \leq i \leq n\}
    \eqstop
  \end{align*}
  Since $\bigcap_{1 \leq i \leq n} g_i \Lambda g_i^{-1} \leq \Lambda$ is a finite index subgroup, $\cR_n$ is ergodic. Therefore $\Omega$ is an ergodic cocycle and $\cR(\Lambda \grpaction{} X) \subset \cR(\Gamma \grpaction{} X)$ is aperiodic. It also follows from Proposition \ref{prop:ergodic-cocycles-give-invariant} that 
  \begin{equation*}
    H(\cR(\Lambda \grpaction{} X) \subset \cR(\Gamma \grpaction{} X))
    = H(\Omega)
    = \ol{\tau_{\Gamma, \Lambda}(\Gamma)}
    = \Gamma \rp \Lambda. \qedhere
  \end{equation*} 
\end{proof}

\section{A stable orbit equivalence rigidity theorem for products of Baumslag-Solitar groups}
\label{sec:rigidity}

In this section, we apply the invariant developed in Section~\ref{sec:invariant} to prove Theorem \ref{thm:main-result}. The next theorem is a straightforward relative version of \cite[Theorem 5.1]{kida11-BS}.

We will need the following terminology. Let $\mathcal R$ be a pmp equivalence relation defined on $(X, \mu)$, $H$ a Polish group together with a Borel action $H \curvearrowright Z$ on a standard Borel space and $\sigma : \mathcal R \to H$ a $1$-cocycle. A Borel map $\varphi : X \to Z$ is said to be $(\mathcal R, \sigma)$-invariant if $\varphi(x) = \sigma(x, y) \varphi(y)$ for almost every $(x, y) \in \mathcal R$. 

To any infinite locally finite simplicial tree $T$, one can associate the {\em topological boundary} of $T$ as the set $\partial T$ of all equivalence classes of infinite geodesic rays in $T$ where two geodesic rays in $T$ are {\em equivalent} if the Hausdorff distance between them is finite. The boundary $\partial T$ is canonically endowed with a Hausdorff topology which makes it a compact space. Moreover, any action $\Gamma \curvearrowright T$ of a discrete countable group by simplicial automorphisms induces an action by homeomorphisms $\Gamma \curvearrowright \partial T$.

\begin{theorem}
  \label{thm:relative-amenability}
  Let $\Gamma$ be any non-amenable discrete countable group acting on a tree $T$ with amenable stabilisers and without edge inversions and $\Lambda$ any discrete countable group. Put $G = \Gamma \times \Lambda$ and denote by $\pi : G \to \Gamma : (\gamma, \lambda) \mapsto \gamma$ the quotient homomorphism.  Let $G \grpaction{} (X, \mu)$ be any free pmp action.  Write $\cR = \cR(G \grpaction{} X)$ for the orbit equivalence relation.  Denote by $\rho : \cR \ra \Gamma$ the $1$-cocycle defined by $\rho(g \cdot x, x) = \pi(g)$ for every $g \in G$ and almost every $x \in X$.

 For every non-negligible Borel subset $U_0 \subset X$ and every amenable commensurated subequivalence relation $\cS \subset \cR |_{U_0}$, there exists an $(\cS, \rho)$-invariant Borel map $\vphi: U_0 \to V(T)$.
\end{theorem}

\begin{proof}
We proceed by contradiction following the lines of \cite[Theorem 5.1]{kida11-BS}. Since any $(\cS, \rho)$-invariant map from a Borel subset of $U_0$ to $V(T)$ can be extended to an $\cS$-invariant subset of $U_0$, we can glue such maps.  So we may assume that there exists a non-negligible $\cS$-invariant Borel subset $U \subset U_0$ such that for every non-negligible Borel subset $W \subset U$ there is no $(\mathcal S |_W, \rho)$-invariant Borel map $\vphi : W \to V(T)$.

Since $\mathcal S |_U \subset \mathcal R |_U$ is amenable, \cite[Proposition 4.14]{kida09} yields an $(\mathcal S, \rho)$-invariant Borel map ${\vphi : U \to \Prob(\partial T)}$.  Denote by $\partial_2 T$ the quotient of $\partial T \times \partial T$ by the action of the symmetric group $\Sym(2)$ that exchanges the two coordinates.  We consider elements of $\partial_2 T$ as subsets of $\partial T$ of size one or two.

  If $\vphi(x)$ is not almost everywhere supported on at most two elements, then, using the fact that $\Gamma$ acts on $T$ without inversions, the proof of \cite[Lemma 5.3]{kida11-BS} shows that $\vphi$ induces an $(\mathcal S, \rho)$-invariant map from a non-negligible subset of $U$ to $V(T)$.  This contradicts the choice of $U$.  So $\vphi$ induces an $(\mathcal S, \rho)$-invariant map $\vphi : U \to \partial_2 T$.  We may moreover choose $\vphi$ so that the measure of the Borel subset $\{x \in U \amid |\supp(\vphi(x))| = 2\}$ is maximal among all the $(\mathcal S, \rho)$-invariant maps $U \to \partial_2 T$.

  Now \cite[Lemma 5.4]{kida11-BS} says that whenever $W \subset U$ is a non-negligible Borel subset and $\mathcal T \subset \mathcal S |_W$ is a finite index subequivalence relation, any $(\mathcal T, \rho)$-invariant Borel map $\psi : W \to \partial_2 T$ satisfies $\psi(x) \subset \vphi(x)$ for almost every $x \in W$.

  Next, we claim that $\vphi : U \to \partial_2 T$ is $(\mathcal R |_U, \rho)$-invariant. Since $\cS |_U \subset \cR |_U$ is commensurated, it suffices to show that for every ${\theta \in \QN_{\mathcal R |_U}(\mathcal S |_U)}$ we have $\vphi(\theta(x)) = \rho(\theta(x), x) \vphi(x)$ for almost every $x \in \dom(\theta)$.  Fix $\theta \in \QN_{\mathcal R |_U}(\mathcal S |_U)$ and define a new Borel map $\psi : \dom(\theta) \to \partial_2 T$ by the formula $\psi(x) = \rho(x, \theta(x)) \vphi(\theta(x))$.  Put $\mathcal T = {(\theta^{-1} \times \theta^{-1})(\mathcal S |_{\rng(\theta)})} \cap \mathcal S |_{\dom (\theta)}$, which is a finite index subequivalence relation of $\mathcal S |_{\dom(\theta)}$.  For almost every $(x, y) \in \mathcal T$, we have
  \begin{align*}
    \rho(y, x)  \psi(x) & = \rho(y, x)  \rho(x, \theta(x))  \vphi(\theta(x)) \\
    & = \rho(y, \theta(x))  \vphi(\theta(x)) \\
    & = \rho(y, \theta(y))  \rho(\theta(y), \theta(x))  \vphi(\theta(x)) \\
    & =  \rho(y, \theta(y))  \vphi(\theta(y)) = \psi(y).
  \end{align*}
  Since $\psi$ is $(\mathcal T, \rho)$-invariant, we get $\psi(x) \subset \vphi(x)$ for almost every $x \in \dom(\theta)$.  Likewise, define $\psi' : \rng(\theta) \to \partial_2 T$ by $\psi'(y) = \rho(y , \theta^{-1}(y)) \vphi(\theta^{-1}(y))$. A similar reasoning shows that $\psi'(y) \subset \vphi(y)$ for almost every $y \in \rng(\theta)$.

  Then for almost every $x \in \dom(\theta)$, we have
  \begin{align*}
    \rho(x, \theta(x)) \vphi(\theta(x))
    & = \psi(x) \\
    & \subset \vphi(x) \\
    & = \rho(x, \theta(x)) \psi'(\theta(x)) \\
    & \subset \rho(x, \theta(x)) \vphi(\theta(x))
  \end{align*}
  Therefore $\vphi(x) = \rho(x, \theta(x)) \vphi(\theta(x))$ for almost every $x \in \dom(\theta)$.  This proves our claim.

  We get in particular that the map $\vphi : U \to \partial_2 T$ is $(\mathcal R(\Gamma \grpaction{} X) |_U, \rho)$-invariant.  We may uniquely extend $\vphi$ to a $(\mathcal R(\Gamma \curvearrowright X), \rho)$-invariant map from the $\mathcal R(\Gamma \curvearrowright X)$-saturation of $U$ (which is equal to $\Gamma U$) into $\partial_2 T$. Since $\rho(\gamma \cdot x , x) = \gamma$ for every $\gamma \in \Gamma$ and almost every $x \in X$, the map $\varphi$ induces a $\Gamma$-invariant probability measure on $\partial_2 T$. However, since $\Gamma$ acts on $T$ with amenable stabilisers, it follows that the action $\Gamma \grpaction{} \partial T$ is topologically amenable (see \cite[Proposition 5.2.1, Lemma 5.2.6]{brownozawa08}).  Therefore the action $\Gamma \grpaction{} \partial T \times \partial T$ is measurably amenable and so is the action $\Gamma \grpaction{} \partial_2 T$ (see \cite[Corollary 3.4]{kida10}). By \cite[Proposition 4.3.3]{zimmer84}, we get that $\Gamma$ is amenable, which is a contradiction.
\end{proof}

\begin{corollary}
  \label{cor:embedding-of-amenable-subrelations}
  Let $n \geq 1$.  For every $i \in \{1, \dots, n\}$, let $\Gamma_i$ be a non-amenable discrete countable group acting on a tree $T_i$ with amenable stabilisers and without edge inversions.  Put $G = {\Gamma_1 \times \cdots \times \Gamma_n}$ and let $G \grpaction{} (X, \mu)$ be a free pmp action.  Let $U_0 \subset X$ be any non-negligible Borel subset and $\cS \subset \cR(G \grpaction{} X)|_{U_0}$ any amenable commensurated subequivalence relation.  Then there exist a non-negligible Borel subset $U \subset U_0$ and vertices $v_i \in V(T_i)$ such that 
  \begin{equation*}
    \cS|_U \subset \cR(\Stab_{\Gamma_1}(v_1) \times \dotsm \times \Stab_{\Gamma_n}(v_n) \grpaction{} X) |_U
    \eqstop
  \end{equation*}
\end{corollary}
\begin{proof}
  We apply Theorem \ref{thm:relative-amenability} inductively.  Assume that there is $i \in \{0, \dotsc, n-1\}$, vertices ${v_1 \in V(T_1)}, \dots, {v_i \in V(T_i)}$ and a non-negligible Borel subset $U_i \subset U_0$ such that
  \begin{equation*}
    \cS|_{U_i} \subset \cR(\Stab_{\Gamma_1}(v_1) \times \dotsm \times \Stab_{\Gamma_i}(v_i) \times \Gamma_{i + 1} \times \dotsm \times \Gamma_{n} \grpaction{} X) |_{U_i}
    \eqstop
  \end{equation*}
  Put
  \begin{equation*}
    \Lambda = \Stab_{\Gamma_1}(v_1) \times \dotsm \times \Stab_{\Gamma_i}(v_i) \times \Gamma_{i + 2} \times \dotsm \times \Gamma_n
  \end{equation*}
  and $\Gamma = \Gamma_{i+1}$. Denote by $\pi : \Gamma \times \Lambda \to \Gamma : (\gamma, \lambda) \mapsto \gamma$ the quotient homomorphism and by $\rho : \cR(\Gamma \times \Lambda \curvearrowright X) \ra \Gamma$ the 1-cocycle defined by $\rho(g \cdot x, x) = \pi(g)$ for every $g \in \Gamma \times \Lambda$ and almost every $x \in X$. Since $\cS|_{U_i}$ is amenable and commensurated by $\cR( \Gamma \times \Lambda \grpaction{} X)|_{U_i}$, Theorem~\ref{thm:relative-amenability} shows that there is an $(\cS|_{U_i}, \rho)$-invariant map $\vphi: U_{i} \ra V(T_{i+1})$.  Take some vertex $v_{i + 1} \in T_{i+1}$ such that $U_{i+1} = \vphi^{-1}(v_{i+1})$ is non-negligible.  Since $\vphi$ is $(\cS|_{U_i}, \rho)$-invariant, it follows that
  \begin{align*}
    \cS|_{U_{i+1}}
    & \subset
    \cR(\Stab_{\Gamma_{i+1}}(v_{i+1}) \times \Lambda \grpaction{} X) |_{U_{i + 1}} \\
    & =
    \cR(\Stab_{\Gamma_1}(v_1) \times \dotsm \times \Stab_{\Gamma_{i+1}}(v_{i+1}) \times \Gamma_{i+2} \times \dotsm \times \Gamma_n \grpaction{} X) |_{U_{i + 1}}
    \eqstop
  \end{align*}
  This completes the induction.  Putting $U = U_n$, we obtain the conclusion.
\end{proof}

We are now ready to prove our main theorem. We will use the notation of Theorem \ref{thm:main-result}.

\begin{proof}[Proof of Theorem \ref{thm:main-result}]
  There are non-negligible Borel subsets $U \subset X$ and $V \subset Y$ and a pmp Borel isomorphism $\vphi: (V, \eta_V) \ra (U, \mu_U)$ such that 
  \begin{equation*}
    (\vphi \times \vphi)(\cR(\BS(p_1, q_1) \times \dotsm \times \BS(p_l,q_l) \grpaction{} Y)|_V)
    =
    \cR(\BS(m_1, n_1) \times \dotsm \times \BS(m_k,n_k) \grpaction{} X)|_U
    \eqstop
  \end{equation*}
  
Put $\Gamma = \BS(m_1, n_1) \times \dotsm \times \BS(m_k,n_k)$. Then $(\vphi \times \vphi)(\cR(\langle b_1 \rangle \times \dotsm \times \langle b_l \rangle \grpaction{} Y)|_V)$ is amenable and commensurated by $\mathcal R(\Gamma \curvearrowright X) |_U$. By Corollary \ref{cor:embedding-of-amenable-subrelations}, for every $i \in \{1, \dotsc, k\}$, there exists a vertex $v_i$ in the Bass-Serre tree of $\BS(m_i, n_i)$ such that, up to taking a non-negligible Borel subset smaller than $U$, we have
  \begin{equation}
    \label{eq:first-inclusion}
    (\vphi \times \vphi)(\cR(\langle b_1 \rangle \times \dotsm \times \langle b_l \rangle \grpaction{} Y)|_{V})
    \subset
    \cR(\Stab(v_1) \times \dotsm \times \Stab(v_k) \grpaction{} X)|_{U}
    \eqstop
  \end{equation}
  
 Put $\Lambda = \BS(p_1, q_1) \times \dotsm \times \BS(p_l, q_l)$. Then ${(\vphi^{-1} \times \vphi^{-1})(\cR(\Stab(v_1) \times \dotsm \times \Stab(v_k) \grpaction{} X)|_{U})}$ is amenable and commensurated by $\mathcal R(\Lambda \curvearrowright Y) |_V$. By Corollary \ref{cor:embedding-of-amenable-subrelations}, for every $j \in \{1, \dotsc, l\}$, there exists a vertex $w_j$ in the Bass-Serre tree of $\BS(p_j, q_j)$ such that, up to taking a non-negligible Borel subset smaller than $V$, we have
  \begin{equation}
    \label{eq:second-inclusion}
    (\vphi^{-1} \times \vphi^{-1})(\cR(\Stab(v_1) \times \dotsm \times \Stab(v_k) \grpaction{} X)|_{U})
    \subset
    \cR(\Stab(w_1) \times \dotsm \times \Stab(w_l) \grpaction{} Y)|_{V}
    \eqstop
  \end{equation}
Combining (\ref{eq:first-inclusion}) and (\ref{eq:second-inclusion}) and writing $\mathcal W = {(\vphi^{-1} \times \vphi^{-1})(\cR(\Stab(v_1) \times \dotsm \times \Stab(v_k) \grpaction{} X)|_{U})}$, we get
\begin{equation}\label{eq:inclusion-equality}
  \cR(\langle b_1 \rangle \times \dotsm \times \langle b_l \rangle \grpaction{} Y)|_{V} \subset \mathcal W  \subset \cR(\Stab(w_1) \times \dotsm \times \Stab(w_l) \grpaction{} Y)|_{V}
  \eqstop
\end{equation}
Since $\cR(\langle b_1 \rangle \cap \Stab(w_1) \times \dotsm \times \langle b_l \rangle \cap \Stab(w_l) \grpaction{} Y)$ is ergodic and 
\begin{align*}
\cR(\langle b_1 \rangle \cap \Stab(w_1) \times \dotsm \times \langle b_l \rangle \cap \Stab(w_l) \grpaction{} Y) |_V \; &= \; \cR(\langle b_1 \rangle \times \dotsm \times \langle b_l \rangle \grpaction{} Y)|_{V} \quad\text{and} \\
\cR(\langle b_1 \rangle \cap \Stab(w_1) \times \dotsm \times \langle b_l \rangle \cap \Stab(w_l) \grpaction{} Y) \; &\subset \; \cR(\langle b_1 \rangle \times \dotsm \times \langle b_l \rangle \grpaction{} Y)
\end{align*}
it follows that 
$$\cR(\langle b_1 \rangle \cap \Stab(w_1) \times \dotsm \times \langle b_l \rangle \cap \Stab(w_l) \grpaction{} Y) = \cR(\langle b_1 \rangle \times \dotsm \times \langle b_l \rangle \grpaction{} Y).$$
Therefore, using freeness, we have $\langle b_j \rangle \cap \Stab(w_j) = \langle b_j \rangle$ and hence $\langle b_j \rangle \subset \Stab(w_j)$ for every $j \in \{1, \dotsc, l\}$. Since $2 \leq |p_j| \leq q_j$, we obtain $\langle b_j \rangle = \Stab(w_j)$ for every $j \in \{1, \dotsc, l\}$. Hence (\ref{eq:inclusion-equality}) yields
\begin{equation*}
 \cR(\Stab(v_1) \times \dotsm \times \Stab(v_k) \grpaction{} X)|_{U} = (\varphi \times \varphi)(\cR(\langle b_1 \rangle \times \dotsm \times \langle b_l \rangle \grpaction{} Y)|_{V})
  \eqstop
\end{equation*}

Since the action of the Baumslag-Solitar group on its Bass-Serre tree is vertex transitive, up to conjugation, we may assume that $\Stab(v_i) = \langle a_i\rangle$ for every $i \in \{1, \dots, k\}$. Hence, we have 
  \begin{equation}
    \label{eq:final-equality}
    \cR(\langle a_1 \rangle \times \dotsm \times \langle a_k \rangle \grpaction{} X)|_{U} =  (\vphi \times \vphi)(\cR(\langle b_1 \rangle \times \dotsm \times \langle b_l \rangle \grpaction{} Y)|_{V}) 
    \eqstop
  \end{equation}

Put $\mathcal T = \cR(\langle a_1 \rangle \times \dotsm \times \langle a_k \rangle \grpaction{} X)$ and $\mathcal S = \cR(\langle b_1 \rangle \times \dotsm \times \langle b_l \rangle \grpaction{} Y)$. By Proposition \ref{prop:aperiodicity-stable}$(i)$, we have 
\begin{equation}\label{eq:identification}
H((\varphi \times \varphi)(\mathcal S |_V) \subset \mathcal R(\Gamma \curvearrowright X) |_U) \cong H(\mathcal S |_V \subset \mathcal R(\Lambda \curvearrowright Y) |_V)
\end{equation}
and by Propositions \ref{prop:aperiodicity-stable}$(ii)$, \ref{prop:quotients-from-actions} and  \ref{prop:relative-profinite-completions-of-products}, we have
\begin{equation}\label{eq:isomorphism1}
H(\mathcal T |_U \subset \mathcal \cR(\Gamma \curvearrowright X) |_U) \cong H(\mathcal T \subset \mathcal \cR(\Gamma \curvearrowright X)) \cong  \rG(m_1, n_1) \times \dotsm \times \rG(m_k, n_k)
\end{equation}
 and
 \begin{equation}\label{eq:isomorphism2}
H(\mathcal S |_V \subset \mathcal \cR(\Lambda \curvearrowright Y) |_V) \cong H(\mathcal S \subset \mathcal \cR(\Lambda \curvearrowright Y)) \cong  \rG(p_1, q_1) \times \dotsm \times \rG(p_l, q_l).
\end{equation}

Combining (\ref{eq:final-equality}), (\ref{eq:identification}), (\ref{eq:isomorphism1}) and (\ref{eq:isomorphism2}), we get 
\begin{equation*}
  \rG(m_1, n_1) \times \dotsm \times \rG(m_k, n_k) \cong \rG(p_1, q_1) \times \dotsm \times \rG(p_l, q_l)
  \eqstop
\end{equation*}
We can now apply Theorem \ref{thm:main-result-on-completions} to obtain the conclusion. This finishes the proof of Theorem \ref{thm:main-result}.
\end{proof}

\bibliographystyle{mybibtexstyle}
\bibliography{operatoralgebras}

{\small
  \parbox[t]{200pt}{Cyril Houdayer\\CNRS - Universit\'e Paris-Est - Marne-la-Vall\'ee \\
    LAMA UMR 8050 \\ 77454 Marne-la-Vall\'ee cedex~2 
\\ France
    \\ {\footnotesize cyril.houdayer@u-pem.fr}}
  \hspace{15pt}
  \parbox[t]{200pt}{Sven Raum\\ RIMS \\
  Kitashirakawa-oiwakecho \\
606-8502 Sakyo-ku, Kyoto \\ 
Japan
    \\ {\footnotesize sven.raum@gmail.com}}
}

\end{document}